\documentclass[a4j,10pt,showkeys]{article}

\usepackage{amsmath,amsthm,amssymb}
\usepackage{geometry}
\usepackage{hyperref}
\usepackage[capitalize]{cleveref}
\usepackage{cancel}
\usepackage{mathtools}
\usepackage{bm}
\usepackage{color}
\usepackage{mathrsfs}

\newtheorem{theo}{Theorem}[section]
\newtheorem{lemm}[theo]{Lemma}

\newtheorem{cor}[theo]{Corollary}
\theoremstyle{definition}
\newtheorem{defi}[theo]{Definition}
\newtheorem{rem}[theo]{Remark}
\newtheorem{assum}{Assumption}

\newcommand{\bA}{\mathbb{A}}

\newcommand{\bC}{\mathbb{C}}

\newcommand{\bE}{\mathbb{E}}
\newcommand{\bF}{\mathbb{F}}

\newcommand{\bN}{\mathbb{N}}

\newcommand{\bP}{\mathbb{P}}

\newcommand{\bR}{\mathbb{R}}
\newcommand{\bS}{\mathbb{S}}
\newcommand{\bT}{\mathbb{T}}

\newcommand{\bX}{\mathbb{X}}

\newcommand{\cF}{\mathcal{F}}

\newcommand{\cI}{\mathcal{I}}

\newcommand{\cL}{\mathcal{L}}

\newcommand{\cP}{\mathcal{P}}

\newcommand{\cS}{\mathcal{S}}
\newcommand{\cT}{\mathcal{T}}
\newcommand{\cU}{\mathcal{U}}

\newcommand{\cX}{\mathcal{X}}
\newcommand{\cY}{\mathcal{Y}}

\newcommand{\sL}{\mathscr{L}}

\newcommand{\ep}{\varepsilon}
\newcommand{\rd}{\,\mathrm{d}}
\newcommand{\op}{\mathrm{op}}
\newcommand{\sym}{\mathrm{sym}}

\newcommand{\1}{\mbox{\rm{1}}\hspace{-0.25em}\mbox{\rm{l}}}
\newcommand{\lint}{\mathalpha{\ltimes}}
\newcommand{\rint}{\mathalpha{\rtimes}}

\providecommand{\keywords}[1]{\textbf{Keywords:} #1}
\makeatletter

\@addtoreset{equation}{section}
\makeatother
\def\widebar{\accentset{{\cc@style\underline{\mskip10mu}}}}
\numberwithin{equation}{section}
\allowdisplaybreaks

\title{Linear-quadratic stochastic Volterra controls I: Causal feedback strategies}

\author{
Yushi Hamaguchi\footnote{Corresponding Author. Graduate School of Engineering Science, Department of Systems Innovation, Osaka University. 1-3, Machikaneyama, Toyonaka, Osaka, Japan (Email: hmgch2950@gmail.com). This author was supported by JSPS KAKENHI Grant Number 22K13958.}
~~ and ~~
Tianxiao Wang\footnote{School of Mathematics, Sichuan University. Chengdu, P.\ R.\ China (Email: wtxiao2014@scu.edu.cn). This author was supported by National Natural Science Foundation of China (No.\ 11971332 and 11931011) and the Science Development Project of Sichuan University under grant 2020SCUNL201.}
}

\begin{document}
\maketitle


\begin{abstract}
In this paper, we formulate and investigate the notion of causal feedback strategies arising in linear-quadratic control problems for stochastic Volterra integral equations (SVIEs) with singular and non-convolution-type coefficients. We show that there exists a unique solution, which we call the causal feedback solution, to the closed-loop system of a controlled SVIE associated with a causal feedback strategy. Furthermore, introducing two novel equations named a Type-II extended backward stochastic Volterra integral equation and a Lyapunov--Volterra equation, we prove a duality principle and a representation formula for a quadratic functional of controlled SVIEs in the framework of causal feedback strategies.
\end{abstract}


\keywords
Linear-quadratic control; stochastic Volterra integral equation; causal feedback strategy.






\section{Introduction}\label{section_Introduction}

Linear-quadratic (LQ) control problems are special classes of optimal control problems described by a linear state dynamics and a quadratic cost functional. In the continuous time setting, the state dynamics is assumed to be governed by a controlled differential/integral equation. In this paper, we study controlled linear stochastic Volterra integral equations (SVIEs) of the following form:
\begin{equation}\label{eq_state}
\begin{split}
	&X(t)=x(t)+\int^t_{t_0}\{A(t,s)X(s)+B(t,s)u(s)+b(t,s)\}\rd s+\int^t_{t_0}\{C(t,s)X(s)+D(t,s)u(s)+\sigma(t,s)\}\rd W(s),\\
	&\hspace{8cm}t\in(t_0,T),
\end{split}
\end{equation}
where $u$ is a control process, $x$ is a given deterministic function called the free term (which is also called the forcing term), $W$ is a Brownian motion, $A,B,C$ and $D$ are matrix-valued deterministic coefficients, and $b$ and $\sigma$ are vector-valued stochastic inhomogeneous terms. The LQ control problem for SVIE \eqref{eq_state}, which we call an \emph{LQ stochastic Volterra control problem}, is a problem to find a control process or a strategy which minimize a quadratic cost functional of both $u$ and $X$.

The controlled SVIE \eqref{eq_state} is a Volterra-type extension of a controlled linear stochastic differential equation (SDE):
\begin{equation}\label{eq_state_SDE}
	\begin{dcases}
	\rd X(s)=\{A(s)X(s)+B(s)u(s)+b(s)\}\rd s+\{C(s)X(s)+D(s)u(s)+\sigma(s)\}\rd W(s),\ s\in[t_0,T],\\
	X(t_0)=x,
	\end{dcases}
\end{equation}
with $x$ being a constant. LQ control problems for SDEs were first studied by Wonham \cite{Wo68} in 1968 and followed by many researchers; see \cite[Chapter 6]{YoZh99} and \cite{SuYo20} for systematic studies and the developments of LQ control theory for SDEs. It is well-known that, in the LQ control problems for SDEs, the optimal control $u$ has the following state-feedback representation:
\begin{equation}\label{eq_feedback_SDE}
	u(t)=\Xi(t)X(t)+v(t),\ t\in[t_0,T],
\end{equation}
where $\Xi$ is a matrix-valued deterministic function, $v$ is a stochastic inhomogeneous term, and $X$ is the optimal state process corresponding to $u$. Therefore, given the strategy $(\Xi,v)$, for each initial condition $(t_0,x)$, we get the following closed-loop system:
\begin{equation}\label{eq_closed-loop_system_SDE}
	\begin{dcases}
	\rd X^{t_0,x}(s)=\{A(s)X^{t_0,x}(s)+B(s)u^{t_0,x}(s)+b(s)\}\rd s\\
	\hspace{3cm}+\{C(s)X^{t_0,x}(s)+D(s)u^{t_0,x}(s)+\sigma(s)\}\rd W(s),\ s\in[t_0,T],\\
	X^{t_0,x}(t_0)=x,\\
	u^{t_0,x}(s)=\Xi(s)X^{t_0,x}(s)+v(s),\ s\in[t_0,T],
	\end{dcases}
\end{equation}
or equivalently the SDE
\begin{equation*}
\begin{dcases}
	\rd X^{t_0,x}(s)=\{(A(s)+B(s)\Xi(s))X^{t_0,x}(s)+B(s)v(s)+b(s)\}\rd s\\
	\hspace{3cm}+\{(C(s)+D(s)\Xi(s))X^{t_0,x}(s)+D(s)v(s)+\sigma(s)\}\rd W(s),\ s\in[t_0,T],\\
	X^{t_0,x}(t_0)=x.
	\end{dcases}
\end{equation*}
We note that the above system is an equation for the optimal state process $X=X^{t_0,x}$, and the optimal control process $u=u^{t_0,x}$ is obtained as the outcome of the strategy $(\Xi,v)$ by inserting the solution $X^{t_0,x}$ into the expression $u^{t_0,x}=\Xi X^{t_0,x}+v$. The closed-loop system \eqref{eq_closed-loop_system_SDE} of the controlled SDE satisfies the flow property: for each $0\leq t_0<t_1<t\leq T$,
\begin{equation*}
	X^{t_0,x}(t)=X^{t_1,X^{t_0,x}(t_1)}(t),\ u^{t_0,x}(t)=u^{t_1,X^{t_0,x}(t_1)}(t).
\end{equation*}
This is also called the time-consistency, which ensures that the strategy $(\Xi,v)$ is meaningful in dynamic sense. It is well-known that the dynamic optimality of the strategy $(\Xi,v)$ is closely related to the well-posedness of a Riccati differential equation and a backward stochastic differential equation (BSDE).

Compared to SDEs, the solutions of SVIEs are not Markovian or even not semimartingales in general. These features of SVIEs make some analysis dramatically difficult since we cannot use some fundamental tools of stochastic calculus such as It\^{o}'s formula directly to compute some functionals of SVIEs. Furthermore, SVIEs do not satisfy the flow property. To illustrate this, consider the following (uncontrolled) stochastic Volterra process:
\begin{equation*}
	X^{t_0,x}(t):=x+\int^t_{t_0}K(t,s)\rd W(s),\ t\in(t_0,T),
\end{equation*}
where $K$ is a deterministic function, and $x$ is an initial constant at time $t_0$. This is in particular the case of a fractional Brownian motion. Due to the two time variables structure of the kernel $K$, $X$ does not satisfy the flow property in the standard sense: for $t_0<t_1<t<T$,
\begin{equation*}
	X^{t_0,x}(t)\neq X^{t_1,X^{t_0,x}(t_1)}(t)\ \ \text{where}\ \ X^{t_1,X^{t_0,x}(t_1)}(t):=X^{t_0,x}(t_1)+\int^t_{t_1}K(t,s)\rd W(s).
\end{equation*}
The lack of the flow property gives rise to the following question: \emph{What is the ``state space'' of SVIEs?} Actually, in the above example, the space of constants $x$ (the space of values of $X^{t_0,x}(t)$) is not suitable for the state space since it is not preserved in dynamic sense. To overcome this difficulty, \cite{AbEu19,HaWo21,ViZh19,WaYoZh20} considered an auxiliary process $\Theta$ defined by
\begin{equation*}
	\Theta^{t_0,x}(s,t):=x+\int^t_{t_0}K(s,r)\rd W(r),\ (s,t)\in\triangle_2(t_0,T),
\end{equation*}
where $\triangle_2(t_0,T):=\{(s,t)\,|\,t_0<t<s<T\}$. Then, for any $t_0<t_1<t<T$, we have the following generalized flow property:
\begin{equation*}
	X^{t_0,x}(t)=X^{t_1,\Theta^{t_0,x}(\cdot,t_1)}(t)\ \ \text{where}\ \ X^{t_1,\Theta^{t_0,x}(\cdot,t_1)}(t):=\Theta^{t_0,x}(t,t_1)+\int^t_{t_1}K(t,s)\rd W(s).
\end{equation*}
In the above expression, $\Theta^{t_0,x}(\cdot,t_1)=(\Theta^{t_0,x}(t,t_1))_{t\in(t_1,T)}$ can be regarded as the free term of $X^{t_1,\Theta^{t_0,x}(\cdot,t_1)}$ evaluated at time $t_1$. We note that $\Theta^{t_0,x}(s,\cdot)=(\Theta^{t_0,x}(s,t))_{t\in[t_0,s]}$ is a semimartingale for each fixed $s\in(t_0,T)$, and it holds that $\Theta^{t_0,x}(t,t)=X^{t_0,x}(t)$. In the dynamic point of view, the space of the auxiliary process $\Theta(\cdot,\cdot)$ can be seen as a state space for the above example. It is worth to mention that, in the theory of mathematical finance, the auxiliary process $\Theta$ corresponds to the so-called forward variance which is a financial index available in the market (see \cite{EuRo18,ViZh19} and references sited therein).

Let us turn to the controlled SVIE \eqref{eq_state}. As mentioned in the above example, the sole space of $X(t)$ is not sufficient for describing the state space. In view of the generalized flow property mentioned above, it is natural to consider the following auxiliary process:
\begin{equation*}
	\Theta(s,t)=x(s)+\int^t_{t_0}\{A(s,r)X(r)+B(s,r)u(r)+b(s,r)\}\rd r+\int^t_{t_0}\{C(s,r)X(r)+D(s,r)u(r)+\sigma(s,r)\}\rd W(r)
\end{equation*}
for $(s,t)\in\triangle_2(t_0,T)$. A particular case of $\Theta(\cdot,\cdot)$ was introduced in \cite{WaT18} to treat LQ stochastic Volterra control problems. In this paper, we call the above process a \emph{forward state process} and regard the pair $(X,\Theta)$ as the (generalized) state of the controlled SVIE \eqref{eq_state}. Introducing the forward state process $\Theta$ to LQ problem also fits into the so-called causal projection approach studied by \cite{HaLiYo21,PrYo96} in deterministic LQ Volterra control problems. From this sense, the forward state $\Theta(s,t)$ can be seen as the causal projection of the original controlled SVIE \eqref{eq_state} which is determined by information of $X$ and $u$ up to the current time $t$. Unlike the SDEs case, the above observation inspires us to consider a feedback control process represented by
\begin{equation*}
	u(t)=\Xi(t)X(t)+\int^T_t\Gamma(s,t)\Theta(s,t)\rd s+v(t)
\end{equation*}
for some deterministic functions $\Xi$ and $\Gamma$ and a stochastic inhomogeneous term $v$. We call the triplet $(\Xi,\Gamma,v)$ a \emph{causal feedback strategy}. By inserting the above expression of the outcome $u$ into the controlled SVIE \eqref{eq_state}, for each $(t_0,x)$, we get the following closed-loop system:
\begin{equation}\label{eq_closed-loop_system}
	\begin{dcases}
	X^{t_0,x}(t)=x(t)+\int^t_{t_0}\{A(t,s)X^{t_0,x}(s)+B(t,s)u^{t_0,x}(s)+b(t,s)\}\rd s\\
	\hspace{3cm}+\int^t_{t_0}\{C(t,s)X^{t_0,x}(s)+D(t,s)u^{t_0,x}(s)+\sigma(t,s)\}\rd W(s),\ t\in(t_0,T),\\
	\Theta^{t_0,x}(s,t)=x(s)+\int^t_{t_0}\{A(s,r)X^{t_0,x}(r)+B(s,r)u^{t_0,x}(r)+b(s,r)\}\rd r\\
	\hspace{3cm}+\int^t_{t_0}\{C(s,r)X^{t_0,x}(r)+D(s,r)u^{t_0,x}(r)+\sigma(s,r)\}\rd W(r),\ (s,t)\in\triangle_2(t_0,T),\\
	u^{t_0,x}(t)=\Xi(t)X^{t_0,x}(t)+\int^T_t\Gamma(s,t)\Theta^{t_0,x}(s,t)\rd s+v(t),\ t\in(t_0,T).
	\end{dcases}
\end{equation}
We note that the above system is a coupled equation for the state process $X^{t_0,x}$ and the forward state process $\Theta^{t_0,x}$, and it is beyond the typical framework of SVIEs studied in existing works. We call the pair $(X^{t_0,x},\Theta^{t_0,x})$ satisfying the above system a \emph{causal feedback solution} to the controlled SVIE \eqref{eq_state} at $(t_0,x)$ corresponding to the causal feedback strategy $(\Xi,\Gamma,v)$.

The purpose of this paper is to formulate and investigate causal feedback solutions of controlled SVIE \eqref{eq_state} systematically. On one hand, the present paper can be seen as a preliminary of our continued paper \cite{HaWa2}, where we go into the main topic of LQ stochastic Volterra control problems and characterize the optimal causal feedback strategy. On the other hand, besides the preliminary, the present paper provides a novel framework for SVIEs which is interesting by its own right in theoretical point of view. The main contributions of this paper are the following three points:
\begin{itemize}
\item[(i)]
We show that there exists a unique causal feedback solution to the controlled SVIE \eqref{eq_state} for any given causal feedback strategy $(\Xi,\Gamma,v)$ and any input condition $(t_0,x)$ (see \cref{theo_SVIE}).
\item[(ii)]
We prove a \emph{duality principle} for the causal feedback solution (see \cref{theo_Type-II_EBSVIE}).
\item[(iii)]
We derive a \emph{representation formula for a quadratic functional} of the causal feedback solution in terms of the free term $x$ (see \cref{theo_Lyapunov--Volterra}).
\end{itemize}

Let us make some comments on the contribution (i) compared to existing works. SVIEs with regular coefficients were first studied by Berger and Mizel \cite{BeMi80,BeMi80+} and followed by \cite{AlNu97,OkZh93,PaPr90}. Studies of SVIEs with singular coefficients (which include fractional SDEs and fractional Brownian motion with Hurst parameter less than $1/2$) can be found in \cite{CoDe01,WaZ08,Zh08,Zh10}, etc. Recently, in \cite{Ha21+++}, the first author of the present paper investigated the algebraic structure of linear SVIEs with singular coefficients and proved a variation of constants formula. In the above works, there are no controls in SVIEs. Optimal control problems of SVIEs were studied by \cite{AgOk15,ChYo07,Ha21+,Ha21++,ShWaYo15,ShWeXi20,WaT18,WaT20,WaTZh17,Yo08} in the open-loop framework, where feedback strategies were not considered. In the special case of SVIEs with completely monotone and convolution-type kernels, several kinds of feedback representations of the optimal controls were investigated by \cite{AbMiPh21,BoCoMa12,CoMa15}. Specifically, Abi Jaber, Miller and Pham \cite{AbMiPh21} studied LQ stochastic Volterra control problems with completely monotone and convolution-type kernels. Based on the so-called Markovian lift approach, which is an infinite-dimensional approach associated with the special structure of the completely monotone kernel, they obtained a kind of a linear feedback representation of the optimal control; see also \cite{AbMiPh21+} for the study on the associated integral operator Riccati equation. Compared to these works, we study linear controlled SVIEs with general (singular and non-convolution-type) coefficients. Our definition of causal feedback strategies, which is motivated from the generalized flow property as mentioned above, is different from the infinite dimensional ones considered in \cite{AbMiPh21,BoCoMa12,CoMa15}. It is worth to mention that Viens and Zhang \cite{ViZh19} proved the functional It\^{o} formula for non-linear SVIEs with non-convolution-type coefficients and applied it to optimal control theory. The corresponding optimal strategy is similar to our definition of causal feedback strategies, but they did not study the well-posedness of the associated closed-loop system \eqref{eq_closed-loop_system}. We emphasize that the closed-loop system \eqref{eq_closed-loop_system} is a coupled system of the original state process $X$ and the forward state process $\Theta$, which is beyond the framework of the existing literature of SVIEs. Our contribution (i) ensures the well-posedness of the closed-loop system \eqref{eq_closed-loop_system}.

Let us illustrate the motivation of the contributions (ii) and (iii). In typical LQ control problems for SDEs, the controlled dynamics of the state process $X$ is described by a linear SDE of the forms of \eqref{eq_state_SDE} in the open-loop framework or \eqref{eq_closed-loop_system_SDE} in the closed-loop framework. In such problems, the key points are to calculate (in terms of given parameters such as controls, initial states and some inhomogeneous terms) a linear functional of the form
\begin{equation}\label{eq_linear_functional_SDE}
	\bE\Bigl[\int^T_{t_0}\langle\psi(t),X(t)\rangle\rd t\Bigr]
\end{equation}
with an adapted process $\psi$ and a quadratic functional of the form
\begin{equation}\label{eq_quadratic_functional_SDE}
	\bE\Bigl[\int^T_{t_0}\langle Q(t)X(t),X(t)\rangle\rd t\Bigr]
\end{equation}
with a symmetric matrix-valued function $Q$. It is well-known that the linear functional \eqref{eq_linear_functional_SDE} for the solution $X$ to a linear SDE can be calculated by using the duality principle in terms of a BSDE, and the quadratic functional \eqref{eq_quadratic_functional_SDE} can be calculated in terms of a Lyapunov differential equation (see \cite{SuYo20}). In contrast, in our continued paper \cite{HaWa2}, we consider LQ stochastic Volterra control problems in the framework of causal feedback strategies, where we are faced with the calculations of a linear functional of the form
\begin{equation}\label{eq_linear_functional_SVIE}
	\bE\Bigl[\int^T_{t_0}\Bigl\{\langle\psi(t),X^{t_0,x}(t)\rangle+\int^T_t\langle\chi(s,t),\Theta^{t_0,x}(s,t)\rangle\rd s\Bigr\}\rd t\Bigr]
\end{equation}
and a quadratic functional of the form
\begin{equation}\label{eq_quadratic_functional_SVIE}
\begin{split}
	&\bE\Bigl[\int^T_{t_0}\Bigl\{\langle Q^{(1)}(t)X^{t_0,x}(t),X^{t_0,x}(t)\rangle+2\int^T_t\langle Q^{(2)}(s,t)X^{t_0,x}(t),\Theta^{t_0,x}(s,t)\rangle\rd s\\
	&\hspace{2cm}+\int^T_t\!\!\int^T_t\langle Q^{(3)}(s_1,s_2,t)\Theta^{t_0,x}(s_2,t),\Theta^{t_0,x}(s_1,t)\rangle\rd s_1\!\rd s_2\Bigr\}\rd t\Bigr].
\end{split}
\end{equation}
Here, $(X^{t_0,x},\Theta^{t_0,x})$ is the causal feedback solution to the controlled SVIE \eqref{eq_state} corresponding to a given causal feedback strategy $(\Xi,\Gamma,v)$, $\psi$ and $\chi$ are stochastic processes, and $Q^{(1)},Q^{(2)}$ and $Q^{(3)}$ are matrix-valued deterministic functions. The analyses of \eqref{eq_linear_functional_SVIE} and \eqref{eq_quadratic_functional_SVIE} serve as basis of the study of LQ stochastic Volterra control problems in \cite{HaWa2}, and these are the main topics of the present paper.

In order to calculate the linear functional \eqref{eq_linear_functional_SVIE}, as a counterpart of a BSDE in the Volterra setting, we introduce a novel class of backward stochastic equations which we call a \emph{Type-II extended backward stochastic Volterra integral equation} (Type-II EBSVIE), see \eqref{eq_Type-II_EBSVIE}. This is a generalization of a class of Type-II backward stochastic Volterra integral equations (Type-II BSVIEs) introduced by Yong \cite{Yo08} to the framework of causal feedback strategies. We prove the well-posedness of Type-II EBSVIEs and derive a duality principle between the causal feedback solution to a controlled SVIE and the adapted solution to a Type-II EBSVIE.

In order to calculate the quadratic functional \eqref{eq_quadratic_functional_SVIE}, we introduce a system of Lyapunov-type Volterra integro-differential equations which we call a \emph{Lyapunov--Volterra equation}, see \eqref{eq_Lyapunov--Volterra}. We prove the well-posedness of Lyapunov--Volterra equations. Furthermore, we provide a representation formula for the quadratic functional \eqref{eq_quadratic_functional_SVIE} of a causal feedback solution to a (homogeneous) controlled SVIE in terms of the free term $x$ and the solution to a Lyapunov--Volterra equation.

As mentioned above, the studies on causal feedback strategies including the duality principle and the representation formula for a quadratic functional play fundamental roles in LQ stochastic Volterra control problems. In our continued paper \cite{HaWa2}, we characterize the optimal causal feedback strategy by means of a \emph{Riccati--Volterra equation}, which is a Riccati-type equation closely related to the Lyapunov--Volterra equation, and a Type-II EBSVIE.

The rest of this paper is organized as follows: In \cref{section_SVIE}, we give a standing assumption of this paper and investigate the novel notion of causal feedback solutions to controlled SVIEs. In \cref{section_Type-II_EBSVIE}, we introduce Type-II EBSVIEs and prove the duality principle. In \cref{section_Lyapunov--Volterra}, we derive Lyapunov--Volterra equations and investigate quadratic functionals of controlled SVIEs. The main theorems \ref{theo_SVIE}, \ref{theo_Type-II_EBSVIE} and \ref{theo_Lyapunov--Volterra} are proved in \cref{appendix} systematically.


\subsection*{Notation}

$(\Omega,\cF,\bP)$ is a complete probability space, and $W$ is a one-dimensional Brownian motion. $\bF=(\cF_t)_{t\geq0}$ denotes the $\bP$-augmented filtration generated by $W$. $\bE[\cdot]$ denotes the expectation, and $\bE_t[\cdot]:=\bE[\cdot|\cF_t]$ denotes the conditional expectation given by $\cF_t$ for each $t\geq0$. Throughout this paper, $\bE[\cdot]^{1/2}$ denotes the square root of the expectation $\bE[\cdot]$, not the expectation of the square root. For each $0\leq t_0<T<\infty$, we define
\begin{align*}
	&\triangle_2(t_0,T):=\{(t,s)\in(t_0,T)^2\,|\,T>t>s>t_0\},\ \text{(a triangle region)}\\
	&\square_3(t_0,T):=\{(s_1,s_2,t)\in(t_0,T)^3\,|\,t<s_1\wedge s_2\}.\ \text{(a square pyramid region)}
\end{align*}
For each matrix $M\in\bR^{d_1\times d_2}$ with $d_1,d_2\in\bN$, $|M|$ denotes the Frobenius norm, and $M^\top\in\bR^{d_2\times d_1}$ denotes the transpose. For each $d\in\bN$, $\bS^d$ denotes the set of $(d\times d)$-symmetric matrices. We define $\bR^d:=\bR^{d \times 1}$, that is, each element of $\bR^d$ is understood as a column vector. We denote by $\langle\cdot,\cdot\rangle$ the usual inner product on a Euclidean space. For each set $\Lambda$, $\1_\Lambda$ denotes the characteristic function.

For each $0\leq t_0<T<\infty$ and $d_1,d_2\in\bN$, we define some spaces of stochastic (and deterministic) processes as follows:
\begin{itemize}
\item
$(L^2_\bF(t_0,T;\bR^{d_1\times d_2}),\|\cdot\|_{L^2_\bF(t_0,T)})$ is the Hilbert space of $\bR^{d_1\times d_2}$-valued, square-integrable and $\bF$-progressively measurable processes on $(t_0,T)$.
\item
$(L^{2,1}_\bF(\triangle_2(t_0,T);\bR^{d_1\times d_2}),\|\cdot\|_{L^{2,1}_\bF(\triangle_2(t_0,T))})$ is the Banach space of $\bR^{d_1\times d_2}$-valued and measurable processes $\xi$ on $\triangle_2(t_0,T)$ such that $\xi(t,\cdot)$ is $\bF$-progressively measurable on $(t_0,t)$ for each $t\in(t_0,T)$, and that $\|\xi\|_{L^{2,1}_\bF(\triangle_2(t_0,T))}<\infty$, where
\begin{equation*}
	\|\xi\|_{L^{2,1}_\bF(\triangle_2(t_0,T))}:=\bE\Bigl[\int^T_{t_0}\Bigl(\int^t_{t_0}|\xi(t,s)|\rd s\Bigr)^2\rd t\Bigr]^{1/2}.
\end{equation*}
\item
$(L^2_\bF(\triangle_2(t_0,T);\bR^{d_1\times d_2}),\|\cdot\|_{L^2_\bF(\triangle_2(t_0,T))})$ is the Hilbert space of $\xi\in L^{2,1}_\bF(\triangle_2(t_0,T);\bR^{d_1\times d_2})$ such that $\|\xi\|_{L^2_\bF(\triangle_2(t_0,T))}<\infty$, where
\begin{equation*}
	\|\xi\|_{L^2_\bF(\triangle_2(t_0,T))}:=\bE\Bigl[\int^T_{t_0}\!\!\int^t_{t_0}|\xi(t,s)|^2\rd s\!\rd t\Bigr]^{1/2}.
\end{equation*}
\item
$(L^2_{\bF,\mathrm{c}}(\triangle_2(t_0,T);\bR^{d_1\times d_2}),\|\cdot\|_{L^2_{\bF,\mathrm{c}}(\triangle_2(t_0,T))})$ is the Banach space of $\xi\in L^2_\bF(\triangle_2(t_0,T);\bR^{d_1\times d_2})$ such that $s\mapsto\xi(t,s)$ is uniformly continuous on $(t_0,t)$ with the limits defined by $\xi(t,t):=\lim_{s\uparrow t}\xi(t,s)$ and $\xi(t,t_0):=\lim_{s\downarrow t_0}\xi(t,s)$ for a.e.\ $t\in(t_0,T)$, a.s., and satisfies $\|\xi\|_{L^2_{\bF,\mathrm{c}}(\triangle_2(t_0,T))}<\infty$, where
\begin{equation*}
	\|\xi\|_{L^2_{\bF,\mathrm{c}}(\triangle_2(t_0,T))}:=\bE\Bigl[\int^T_{t_0}\sup_{s\in[t_0,t]}|\xi(t,s)|^2\rd t\Bigr]^{1/2}.
\end{equation*}
\item
$(L^\infty(t_0,T;\bR^{d_1\times d_2}),\|\cdot\|_{L^\infty(t_0,T)})$ is the Banach space of  $\bR^{d_1\times d_2}$-valued essentially bounded measurable functions on $(t_0,T)$.
\item
For $\bT=(t_0,T),(t_0,T)^2,\triangle_2(t_0,T)$ and $\square_3(t_0,T)$, $(L^2(\bT;\bR^{d_1\times d_2}),\|\cdot\|_{L^2(\bT)})$ is the Hilbert space of $\bR^{d_1\times d_2}$-valued and square-integrable deterministic functions on $\bT$.
\item
$(L^{2,1}(\triangle_2(t_0,T);\bR^{d_1\times d_2}),\|\cdot\|_{L^{2,1}(\triangle_2(t_0,T))})$ is the Banach space of $\bR^{d_1\times d_2}$-valued deterministic functions $f$ on $\triangle_2(t_0,T)$ such that $\|f\|_{L^{2,1}(\triangle_2(t_0,T))}<\infty$, where
\begin{equation*}
	\|f\|_{L^{2,1}(\triangle_2(t_0,T))}:=\Bigl(\int^T_{t_0}\Bigl(\int^t_{t_0}|f(t,s)|\rd s\Bigr)^2\rd t\Bigr)^{1/2}.
\end{equation*}
\item
$(L^{2,2,1}(\square_3(t_0,T);\bR^{d_1\times d_2}),\|\cdot\|_{L^{2,2,1}(\square_3(t_0,T))})$ is the Banach space of $\bR^{d_1\times d_2}$-valued deterministic functions $f$ on $\square_3(t_0,T)$ such that $\|f\|_{L^{2,2,1}(\square_3(t_0,T))}<\infty$, where
\begin{equation*}
	\|f\|_{L^{2,2,1}(\square_3(t_0,T))}:=\Bigl(\int^T_{t_0}\!\!\int^T_{t_0}\Bigl(\int^{s_1\wedge s_2}_{t_0}|f(s_1,s_2,t)|\rd t\Bigr)^2\rd s_1\!\rd s_2\Bigr)^{1/2}<\infty.
\end{equation*}
\item
$L^{2,2,1}_\sym(\square_3(t_0,T);\bR^{d\times d})$ is the set of $f\in L^{2,2,1}(\square_3(t_0,T);\bR^{d\times d})$ such that $f(s_1,s_2,t)=f(s_2,s_1,t)^\top$ for a.e.\ $(s_1,s_2,t)\in\square_3(0,T)$. It is easy to see that $L^{2,2,1}_\sym(\square_3(t_0,T);\bR^{d\times d})$ is a closed subspace of $L^{2,2,1}(\square_3(t_0,T);\bR^{d_1\times d_2})$.
\item
$\sL^2(\triangle_2(t_0,T);\bR^{d_1\times d_2})$ is the set of $f\in L^2(\triangle_2(t_0,T);\bR^{d_1\times d_2})$ satisfying the following two conditions:
\begin{itemize}
\item
it holds that
\begin{equation*}
	\|f\|_{\sL^2(\triangle_2(t_0,T))}:=\underset{t\in(t_0,T)}{\mathrm{ess\,sup}}\Bigl(\int^T_t|f(s,t)|^2\rd s\Bigr)^{1/2}<\infty;
\end{equation*}
\item
for any $\ep>0$, there exists a finite partition $\{U_i\}^m_{i=0}$ of $(t_0,T)$ with $t_0=U_0<U_1<\cdots<U_m=T$ such that
\begin{equation*}
	\underset{t\in(U_i,U_{i+1})}{\mathrm{ess\,sup}}\Bigl(\int^{U_{i+1}}_t|f(s,t)|^2\rd s\Bigr)^{1/2}<\ep
\end{equation*}
for each $i\in\{0,1,\dots,m-1\}$.
\end{itemize}
It is easy to see that $(\sL^2(\triangle_2(t_0,T);\bR^{d_1\times d_2}),\|\cdot\|_{\sL^2(\triangle_2(t_0,T))})$ is a Banach space.
\end{itemize}


\begin{rem}\label{rem_sL^2}
The space $\sL^2(\triangle_2(t_0,T);\bR^{d_1\times d_2})$ includes important classes of singular kernels appearing in the study of SVIEs. On one hand, convolution-type kernels of the form $f(t,s)=f_1(t-s)f_2(s)$, $(t,s)\in\triangle_2(0,T)$, with $f_1\in L^2(0,T;\bR)$ and $f_2\in L^\infty(0,T;\bR^{d_1\times d_2})$ is in $\sL^2(\triangle_2(t_0,T);\bR^{d_1\times d_2})$. On the other hand, if $f$ satisfies the $L^p$-integrability
\begin{equation*}
	 \underset{t\in(t_0,T)}{\mathrm{ess\,sup}}\Bigl(\int^T_t|f(s,t)|^p\rd s\Bigr)^{1/p}<\infty
\end{equation*}
for some $p>2$, then it is in $\sL^2(\triangle_2(t_0,T);\bR^{d_1\times d_2})$ too.
\end{rem}

Throughout this paper, $d\in\bN$ represents the dimension of state processes, and $\ell\in\bN$ represents the dimension of control processes. We fix a finite terminal time $T\in(0,\infty)$.


\section{Causal feedback solutions}\label{section_SVIE}

We define the set of input conditions by $\cI:=\{(t_0,x)\,|\,t_0\in[0,T),\,x\in L^2(t_0,T;\bR^d)\}$ and the set of control processes by $\cU(t_0,T):=L^2_\bF(t_0,T;\bR^\ell)$. For each input condition $(t_0,x)\in \cI$ and control $u\in\cU(t_0,T)$, consider the controlled linear SVIE \eqref{eq_state}. The following is the standing assumption of this paper.


\begin{assum}\label{assum_coefficients}
\begin{itemize}
\item
The coefficients: $A\in L^2(\triangle_2(0,T);\bR^{d\times d})$, $B\in L^2(\triangle_2(0,T);\bR^{d\times\ell})$, $C\in\sL^2(\triangle_2(0,T);\bR^{d\times d})$, $D\in\sL^2(\triangle_2(0,T);\bR^{d\times\ell})$.
\item
The inhomogeneous terms: $b\in L^{2,1}_\bF(\triangle_2(0,T);\bR^d)$, $\sigma\in L^2_\bF(\triangle_2(0,T);\bR^d)$.
\end{itemize}
\end{assum}


\begin{rem}
In the standing assumption, the coefficients and the inhomogeneous terms are singular and of non-convolution-types. For example, $A(t,s)$ is allowed to diverge as $s\uparrow t$, and the same is true for $B,C,D,b$ and $\sigma$. Our framework is more general than \cite{ChYo07} (where the coefficients are of non-convolution-types, but $B,C$ and $D$ are essentially regular, and the inhomogeneous terms $b$ and $\sigma$ do not appear) and \cite{AbMiPh21,AbMiPh21+} (where the coefficients are singular, but they are of convolution-types with completely monotone kernels, and the inhomogeneous terms are deterministic). See \cref{rem_sL^2}. It is also worth to mention that the assumptions for the coefficients $C$ and $D$ being in $\sL^2(\triangle_2(t_0,T);\bR^{d_1\times d_2})$ fit into the framework of the so-called $\star$-Volterra kernels introduced in \cite{Ha21+++}.
\end{rem}

Based on the discussion in the introductory section, we define the notions of causal feedback strategies and the associated causal feedback solutions of linear controlled SVIEs.


\begin{defi}
Each triplet $(\Xi,\Gamma,v)\in\cS(0,T):=L^\infty(0,T;\bR^{\ell\times d})\times L^2(\triangle_2(0,T);\bR^{\ell\times d})\times\cU(0,T)$ is called a \emph{casual feedback strategy}. For each $(\Xi,\Gamma,v)\in\cS(0,T)$ and $(t_0,x)\in\cI$, we say that a triplet $(X^{t_0,x},\Theta^{t_0,x},u^{t_0,x})\in L^2_\bF(t_0,T;\bR^d)\times L^2_{\bF,\mathrm{c}}(\triangle_2(t_0,T);\bR^d)\times\cU(t_0,T)$ is a \emph{casual feedback solution} of controlled SVIE \eqref{eq_state} at $(t_0,x)$ corresponding to $(\Xi,\Gamma,v)$ if it satisfies the closed-loop system \eqref{eq_closed-loop_system}. We sometimes call the pair $(X^{t_0,x},\Theta^{t_0,x})$ the causal feedback solution for simplicity. The control process $u^{t_0,x}\in\cU(t_0,T)$ is called the \emph{outcome} of the causal feedback strategy $(\Xi,\Gamma,v)$ at $(t_0,x)$, and we write $(\Xi,\Gamma,v)[t_0,x](t):=u^{t_0,x}(t)$.
\end{defi}

\begin{rem}
Note that if $(X^{t_0,x},\Theta^{t_0,x})\in L^2_\bF(t_0,T;\bR^d)\times L^2_{\bF,\mathrm{c}}(\triangle_2(t_0,T);\bR^d)$ satisfies the closed-loop system \eqref{eq_closed-loop_system} for a given causal feedback strategy $(\Xi,\Gamma,v)\in\cS(0,T)$ and an input condition $(t_0,x)\in\cI$, then the outcome $u^{t_0,x}=(\Xi,\Gamma,v)[t_0,x]$ is automatically in $\cU(t_0,T)$. We emphasize that the causal feedback strategy $(\Xi,\Gamma,v)$ is chosen to be independent of the input condition $(t_0,x)$, while the causal feedback solution $(X^{t_0,x},\Theta^{t_0,x})$ and the outcome $u^{t_0,x}=(\Xi,\Gamma,v)[t_0,x]$ depend on $(t_0,x)$.
\end{rem}

We will also consider the homogeneous version of the controlled SVIE \eqref{eq_state}, where the inhomogeneous terms $b$ and $\sigma$ vanish. In this case, the controlled SVIE \eqref{eq_state} becomes
\begin{equation}\label{eq_state_0}
	X_0(t)=x(t)+\int^t_{t_0}\{A(t,s)X_0(s)+B(t,s)u(s)\}\rd s+\int^t_{t_0}\{C(t,s)X_0(s)+D(t,s)u(s)\}\rd W(s),\ t\in(t_0,T).
\end{equation}
For each causal feedback strategy $(\Xi,\Gamma,v)\in\cS(0,T)$ and input condition $(t_0,x)\in\cI$, the corresponding causal feedback solution $(X^{t_0,x}_0,\Theta^{t_0,x}_0)\in L^2_\bF(t_0,T;\bR^d)\times L^2_{\bF,\mathrm{c}}(\triangle_2(t_0,T);\bR^d)$ of the homogeneous controlled SVIE \eqref{eq_state_0} satisfies
\begin{equation*}
	\begin{dcases}
	X^{t_0,x}_0(t)=x(t)+\int^t_{t_0}\{A(t,s)X^{t_0,x}_0(s)+B(t,s)u^{t_0,x}_0(s)\}\rd s\\
	\hspace{2cm}+\int^t_{t_0}\{C(t,s)X^{t_0,x}_0(s)+D(t,s)u^{t_0,x}_0(s)\}\rd W(s),\ \ t\in(t_0,T),\\
	\Theta^{t_0,x}_0(s,t)=x(s)+\int^t_{t_0}\{A(s,r)X^{t_0,x}_0(r)+B(s,r)u^{t_0,x}_0(r)\}\rd r\\
	\hspace{2cm}+\int^t_{t_0}\{C(s,r)X^{t_0,x}_0(r)+D(s,r)u^{t_0,x}_0(r)\}\rd W(r),\ \ (s,t)\in\triangle_2(t_0,T),\\
	u^{t_0,x}_0(t)=\Xi(t)X^{t_0,x}_0(t)+\int^T_t\Gamma(s,t)\Theta^{t_0,x}_0(s,t)\rd s+v(t),\ t\in(t_0,T).
	\end{dcases}
\end{equation*}
We denote the outcome by $(\Xi,\Gamma,v)^0[t_0,x]:=u^{t_0,x}_0\in\cU(t_0,T)$.

The following is concerned with the existence and uniqueness of the causal feedback solution. This is the first main result of this paper.


\begin{theo}\label{theo_SVIE}
For each causal feedback strategy $(\Xi,\Gamma,v)\in\cS(0,T)$ and each input condition $(t_0,x)\in\cI$, the controlled SVIE \eqref{eq_state} has a unique causal feedback solution $(X^{t_0,x},\Theta^{t_0,x},u^{t_0,x})\in L^2_\bF(t_0,T;\bR^d)\times L^2_{\bF,\mathrm{c}}(\triangle_2(t_0,T);\bR^d)\times\cU(t_0,T)$. Furthermore, there exists a constant $K>0$ depending only on $A,B,C,D,\Xi,\Gamma$ such that
\begin{equation}\label{eq_closed-loop_estimate}
\begin{split}
	 &\|X^{t_0,x}\|_{L^2_\bF(t_0,T)}+\|\Theta^{t_0,x}\|_{L^2_{\bF,\mathrm{c}}(\triangle_2(t_0,T))}+\|u^{t_0,x}\|_{L^2_\bF(t_0,T)}\\
	&\leq K(\|x\|_{L^2(t_0,T)}+\|b\|_{L^{2,1}_\bF(\triangle_2(t_0,T))}+\|\sigma\|_{L^2_\bF(\triangle_2(t_0,T))}+\|v\|_{L^2_\bF(t_0,T)}).
\end{split}
\end{equation}
\end{theo}


\begin{proof}
See \cref{appendix_SVIE}.
\end{proof}

Let us further observe the structure of the causal feedback solution. Let a causal feedback strategy $(\Xi,\Gamma,v)\in\cS(0,T)$ be fixed. For each input condition $(t_0,x)\in\cI$, let $(X^{t_0,x},\Theta^{t_0,x},u^{t_0,x})$ be the corresponding causal feedback solution to the controlled SVIE \eqref{eq_state}. Clearly, $(X^{t_0,x},\Theta^{t_0,x},u^{t_0,x})$ does not depend on $\Xi(t)$ for $t\in(0,t_0)$, $\Gamma(s,t)$ for $s\in(t,T)$ and $t\in(0,t_0)$, or $v(t)$ for $t\in(0,t_0)$. Furthermore, for any $t_1\in(t_0,T)$, we have
\begin{equation*}
	\begin{dcases}
	X^{t_0,x}(t)=\Theta^{t_0,x}(t,t_1)+\int^t_{t_1}\{A(t,s)X^{t_0,x}(s)+B(t,s)u^{t_0,x}(s)+b(t,s)\}\rd s\\
	\hspace{3cm}+\int^t_{t_1}\{C(t,s)X^{t_0,x}(s)+D(t,s)u^{t_0,x}(s)+\sigma(t,s)\}\rd W(s),\ t\in(t_1,T),\\
	\Theta^{t_0,x}(s,t)=\Theta^{t_0,x}(s,t_1)+\int^t_{t_1}\{A(s,r)X^{t_0,x}(r)+B(s,r)u^{t_0,x}(r)+b(s,r)\}\rd r\\
	\hspace{3cm}+\int^t_{t_1}\{C(s,r)X^{t_0,x}(r)+D(s,r)u^{t_0,x}(r)+\sigma(s,r)\}\rd W(r),\ (s,t)\in\triangle_2(t_1,T),\\
	u^{t_0,x}(t)=\Xi(t)X^{t_0,x}(t)+\int^T_t\Gamma(s,t)\Theta^{t_0,x}(s,t)\rd s+v(t),\ t\in(t_1,T).
	\end{dcases}
\end{equation*}
By the uniqueness of the causal feedback solution on $(t_1,T)$, we see that
\begin{equation*}
	\begin{dcases}
	X^{t_0,x}(t)=X^{t_1,\Theta^{t_0,x}(\cdot,t_1)}(t),\ t\in(t_1,T),\\
	\Theta^{t_0,x}(s,t)=\Theta^{t_1,\Theta^{t_0,x}(\cdot,t_1)}(s,t),\ (s,t)\in\triangle_2(t_1,T),\\
	u^{t_0,x}(t)=u^{t_1,\Theta^{t_0,x}(\cdot,t_1)}(t),\ t\in(t_1,T).
	\end{dcases}
\end{equation*}
This is the generalized flow property of the causal feedback solution.


\section{Type-II EBSVIEs and duality principle}\label{section_Type-II_EBSVIE}

In this section, for a given causal feedback strategy $(\Xi,\Gamma,v)\in\cS(0,T)$ and $(\psi,\chi)\in L^2_\bF(0,T;\bR^d)\times L^{2,1}_\bF(\triangle_2(0,T);\bR^d)$, we investigate the linear functional \eqref{eq_linear_functional_SVIE} which we rewrite here for readers' convenience:
\begin{equation*}
	\bE\Bigl[\int^T_{t_0}\Bigl\{\langle\psi(t),X^{t_0,x}(t)\rangle+\int^T_t\langle\chi(s,t),\Theta^{t_0,x}(s,t)\rangle\rd s\Bigr\}\rd t\Bigr].
\end{equation*}
Here, $(X^{t_0,x},\Theta^{t_0,x})\in L^2_\bF(t_0,T;\bR^d)\times L^2_{\bF,\mathrm{c}}(\triangle_2(t_0,T);\bR^d)$ is the causal feedback solution to the controlled SVIE \eqref{eq_state} at $(t_0,x)\in\cI$ corresponding to $(\Xi,\Gamma,v)$. By the linearity of $(X^{t_0,x},\Theta^{t_0,x})$ with respect to the tuple $(x,b,\sigma,v)\in L^2(t_0,T;\bR^d)\times L^{2,1}_\bF(\triangle_2(t_0,T);\bR^d)\times L^2_\bF(\triangle_2(t_0,T);\bR^d)\times L^2_\bF(t_0,T;\bR^d)$, together with the estimate \eqref{eq_closed-loop_estimate}, the above term can be seen as a bounded linear functional $\Psi$ on the product space:
\begin{equation*}
	\Psi(x,b,\sigma,v):=\bE\Bigl[\int^T_{t_0}\Bigl\{\langle\psi(t),X^{t_0,x}(t)\rangle+\int^T_t\langle\chi(s,t),\Theta^{t_0,x}(s,t)\rangle\rd s\Bigr\}\rd t\Bigr].
\end{equation*}
Note that the bounded linear functional $\Psi$ is determined by $A,B,C,D,\Xi,\Gamma$ and $(\psi,\chi)$. Our purpose is to derive the dual representation of $\Psi$.

To do so, we introduce the following (linear) \emph{Type-II extended backward stochastic Volterra integral equation} (Type-II EBSVIE):
\begin{equation}\label{eq_Type-II_EBSVIE}
	\begin{dcases}
	\rd\eta(t,s)=-\Bigl\{\chi(t,s)+\Gamma(t,s)^\top\int^T_sB(r,s)^\top\eta(r,s)\rd r+\Gamma(t,s)^\top\int^T_sD(r,s)^\top\zeta(r,s)\rd r\Bigr\}\rd s\\
	\hspace{3cm}+\zeta(t,s)\rd W(s),\ (t,s)\in\triangle_2(0,T),\\
	\eta(t,t)=\psi(t)+\int^T_t(A+B\triangleright\Xi)(r,t)^\top\eta(r,t)\rd r+\int^T_t(C+D\triangleright\Xi)(r,t)^\top\zeta(r,t)\rd r,\ t\in(0,T),
	\end{dcases}
\end{equation}
with $A+B\triangleright\Xi\in L^2(\triangle_2(0,T);\bR^{d\times d})$ and $C+D\triangleright\Xi\in\sL^2(\triangle_2(0,T);\bR^{d\times d})$ defined by
\begin{equation*}
	(A+B\triangleright\Xi)(t,s):=A(t,s)+B(t,s)\Xi(s),\ (C+D\triangleright\Xi)(t,s):=C(t,s)+D(t,s)\Xi(s)
\end{equation*}
for $(t,s)\in\triangle_2(0,T)$.


\begin{defi}
We say that a pair $(\eta,\zeta)\in L^2_{\bF,\mathrm{c}}(\triangle_2(0,T);\bR^d)\times L^2_\bF(\triangle_2(0,T);\bR^d)$ is an \emph{adapted solution} to the Type-II EBSVIE \eqref{eq_Type-II_EBSVIE} if it satisfies
\begin{equation*}
	\begin{dcases}
	\eta(t,\theta)=\eta(t,t)+\int^t_\theta\Bigl\{\chi(t,s)+\Gamma(t,s)^\top\int^T_s B(r,s)^\top\eta(r,s)\rd r+\Gamma(t,s)^\top\int^T_s D(r,s)^\top\zeta(r,s)\rd r\Bigr\}\rd s\\
	\hspace{4cm}-\int^t_\theta\zeta(t,s)\rd W(s),\\
	\eta(t,t)=\psi(t)+\int^T_t(A+B\triangleright\Xi)(r,t)^\top\eta(r,t)\rd r+\int^T_t(C+D\triangleright\Xi)(r,t)^\top\zeta(r,t)\rd r,
	\end{dcases}
\end{equation*}
for a.e.\ $t\in(0,T)$ and any $\theta\in[0,t]$, a.s.
\end{defi}


\begin{rem}
The Type-II EBSVIE \eqref{eq_Type-II_EBSVIE} is a generalization of a class of linear Type-II BSVIEs introduced by Yong \cite{Yo08}. Indeed, if $\chi=0$ and $(\Xi,\Gamma)=(0,0)$, then we see from the first equation in \eqref{eq_Type-II_EBSVIE} that
\begin{equation*}
	\eta(t,s)=\bE_s[\eta(t)],\ \eta(t)=\bE[\eta(t)]+\int^t_0\zeta(t,s)\rd W(s),
\end{equation*}
for $(t,s)\in\triangle_2(0,T)$, where $\eta(t):=\eta(t,t)$. Thus, the second equation in \eqref{eq_Type-II_EBSVIE} becomes
\begin{equation*}
	\eta(t)=\psi(t)+\int^T_t\{A(r,t)^\top\bE_t[\eta(r)]+C(r,t)^\top\zeta(r,t)\}\rd r,\ t\in(0,T),
\end{equation*}
or equivalently, by means of the martingale representation theorem,
\begin{equation*}
	\eta(t)=\psi(t)+\int^T_t\{A(r,t)^\top\eta(r)+C(r,t)^\top\zeta(r,t)\}\rd r-\int^T_t\zeta(t,r)\rd W(r),\ t\in(0,T),
\end{equation*}
for a suitable choice of $\zeta(t,r)$ for $(r,t)\in\triangle_2(0,T)$. This is a linear Type-II BSVIE. In this case, by \cite{Ha21+++,Yo08}, the following duality principle holds:
\begin{align*}
	&\bE\Bigl[\int^T_{t_0}\langle\psi(t),X(t)\rangle\rd t\Bigr]=\bE\Bigl[\int^T_{t_0}\langle\eta(t),\varphi^{x,b,\sigma,v}(t)\rangle\rd t\Bigr]\\
	&=\int^T_{t_0}\langle\bE[\eta(t)],x(t)\rangle\rd t+\bE\Bigl[\int^T_{t_0}\Bigl\{\int^T_t\langle\bE_t[\eta(s)],b(s,t)\rangle\rd s+\int^T_t\langle\zeta(s,t),\sigma(s,t)\rangle\rd s\Bigr\}\rd t\Bigr]\\
	&\hspace{1cm}+\bE\Bigl[\int^T_{t_0}\Bigl\langle\int^T_t\{B(s,t)^\top\bE_t[\eta(s)]+D(s,t)^\top\zeta(s,t)\}\rd s,v(t)\Bigr\rangle\rd t\Bigr],
\end{align*}
where $X\in L^2_\bF(t_0,T;\bR^d)$ is the solution to the SVIE
\begin{equation*}
	X(t)=\varphi^{x,b,\sigma,v}(t)+\int^t_{t_0}A(t,s)X(s)\rd s+\int^t_{t_0}C(t,s)X(s)\rd W(s),\ t\in(t_0,T),
\end{equation*}
with the free term $\varphi^{x,b,\sigma,v}(t)\in L^2_\bF(t_0,T;\bR^d)$ defined by
\begin{equation*}
	\varphi^{x,b,\sigma,v}(t):=x(t)+\int^t_{t_0}\{B(t,s)v(s)+b(t,s)\}\rd s+\int^t_{t_0}\{D(t,s)v(s)+\sigma(t,s)\}\rd W(s),\ t\in(t_0,T).
\end{equation*}
\end{rem}

The following theorem generalizes the above duality principle to the framework of causal feedback solutions of controlled SVIEs. This is the second main result of the present paper.


\begin{theo}\label{theo_Type-II_EBSVIE}
Let $(\Xi,\Gamma)\in L^\infty(0,T;\bR^{\ell\times d})\times L^2(\triangle_2(0,T);\bR^{\ell\times d})$ be fixed. For any $(\psi,\chi)\in L^2_\bF(0,T;\bR^d)\times L^{2,1}_\bF(\triangle_2(0,T);\bR^d)$, there exists a unique adapted solution $(\eta,\zeta)\in L^2_{\bF,\mathrm{c}}(\triangle_2(0,T);\bR^d)\times L^2_\bF(\triangle_2(0,T);\bR^d)$ to the Type-II EBSVIE \eqref{eq_Type-II_EBSVIE}. Furthermore, for any $v\in\cU(0,T)$ and $(t_0,x)\in\cI$, the following \emph{duality principle} holds:
\begin{equation}\label{eq_duality}
\begin{split}
	&\bE\Bigl[\int^T_{t_0}\Bigl\{\langle\psi(t),X^{t_0,x}(t)\rangle+\int^T_t\langle\chi(s,t),\Theta^{t_0,x}(s,t)\rangle\rd s\Bigr\}\rd t\Bigr]\\
	&=\int^T_{t_0}\langle\bE[\eta(t,t_0)],x(t)\rangle\rd t+\bE\Bigl[\int^T_{t_0}\Bigl\{\int^T_t\langle\eta(s,t),b(s,t)\rangle\rd s+\int^T_t\langle\zeta(s,t),\sigma(s,t)\rangle\rd s\Bigr\}\rd t\Bigr]\\
	&\hspace{1cm}+\bE\Bigl[\int^T_{t_0}\Bigl\langle\int^T_t\{B(s,t)^\top\eta(s,t)+D(s,t)^\top\zeta(s,t)\}\rd s,v(t)\Bigr\rangle\rd t\Bigr],
\end{split}
\end{equation}
where $(X^{t_0,x},\Theta^{t_0,x})\in L^2_\bF(t_0,T;\bR^d)\times L^2_{\bF,\mathrm{c}}(\triangle_2(t_0,T);\bR^d)$ is the causal feedback solution to the controlled SVIE \eqref{eq_state} at $(t_0,x)$ corresponding to the causal feedback strategy $(\Xi,\Gamma,v)\in\cS(0,T)$.
\end{theo}


\begin{proof}
See \cref{appendix_Type-II_EBSVIE}.
\end{proof}


\section{Lyapunov--Volterra equations and quadratic functionals}\label{section_Lyapunov--Volterra}

In this section, we investigate the quadratic functional \eqref{eq_quadratic_functional_SVIE} with respect to the homogeneous controlled SVIE \eqref{eq_state_0}. Specifically, for each fixed $(\Xi,\Gamma)\in L^\infty(0,T;\bR^{\ell\times d})\times L^2(\triangle_2(0,T);\bR^{\ell\times d})$ and $(Q^{(1)},Q^{(2)},Q^{(3)})\in L^\infty(0,T;\bS^d)\times L^2(\triangle_2(0,T);\bR^{d\times d})\times L^{2,2,1}_\sym(\square_3(0,T);\bR^{d\times d})$, consider the following term:
\begin{equation}\label{eq_quadratic_functional_0}
\begin{split}
	&\bE\Bigl[\int^T_{t_0}\Bigl\{\langle Q^{(1)}(t)X^{t_0,x}_0(t),X^{t_0,x}_0(t)\rangle+2\int^T_t\langle Q^{(2)}(s,t)X^{t_0,x}_0(t),\Theta^{t_0,x}_0(s,t)\rangle\rd s\\
	&\hspace{1.5cm}+\int^T_t\!\!\int^T_t\langle Q^{(3)}(s_1,s_2,t)\Theta^{t_0,x}_0(s_2,t),\Theta^{t_0,x}_0(s_1,t)\rangle\rd s_1\!\rd s_2\Bigr\}\rd t\Bigr],
\end{split}
\end{equation}
where $(X^{t_0,x}_0,\Theta^{t_0,x}_0)\in L^2_\bF(t_0,T;\bR^d)\times L^2_{\bF,\mathrm{c}}(\triangle_2(t_0,T);\bR^d)$ is the causal feedback solution to the homogeneous controlled SVIE \eqref{eq_state_0} at $(t_0,x)\in\cI$ corresponding to the causal feedback strategy $(\Xi,\Gamma,0)\in\cS(0,T)$. \eqref{eq_quadratic_functional_0} can be written as $\Phi^{t_0}(x,x)$ with
\begin{align*}
	\Phi^{t_0}(x,y)&:=\bE\Bigl[\int^T_{t_0}\Bigl\{\langle Q^{(1)}(t)X^{t_0,x}_0(t),X^{t_0,y}_0(t)\rangle\\
	&\hspace{2cm}+\int^T_t\langle Q^{(2)}(s,t)X^{t_0,x}_0(t),\Theta^{t_0,y}_0(s,t)\rangle\rd s+\int^T_t\langle Q^{(2)}(s,t)X^{t_0,y}_0(t),\Theta^{t_0,x}_0(s,t)\rangle\rd s\\
	&\hspace{2cm}+\int^T_t\!\!\int^T_t\langle Q^{(3)}(s_1,s_2,t)\Theta^{t_0,x}_0(s_2,t),\Theta^{t_0,y}_0(s_1,t)\rangle\rd s_1\!\rd s_2\Bigr\}\rd t\Bigr]
\end{align*}
for $x,y\in L^2(t_0,T;\bR^d)$. For each fixed $t_0\in[0,T)$, noting the linearity of $(X^{t_0,x}_0,\Theta^{t_0,x}_0)$ with respect to $x\in L^2(t_0,T;\bR^d)$, the estimate \eqref{eq_closed-loop_estimate} and the symmetricity of $(Q^{(1)},Q^{(2)},Q^{(3)})$, we see that $\Phi^{t_0}$ is a bounded symmetric bilinear form on the Hilbert space $L^2(t_0,T;\bR^d)$. Therefore, there exists a unique self-adjoint bounded linear operator $\cP^{t_0}$ on $L^2(t_0,T;\bR^d)$ such that
\begin{align*}
	&\bE\Bigl[\int^T_{t_0}\Bigl\{\langle Q^{(1)}(t)X^{t_0,x}_0(t),X^{t_0,x}_0(t)\rangle+2\int^T_t\langle Q^{(2)}(s,t)X^{t_0,x}_0(t),\Theta^{t_0,x}_0(s,t)\rangle\rd s\\
	&\hspace{1.5cm}+\int^T_t\!\!\int^T_t\langle Q^{(3)}(s_1,s_2,t)\Theta^{t_0,x}_0(s_2,t),\Theta^{t_0,x}_0(s_1,t)\rangle\rd s_1\!\rd s_2\Bigr\}\rd t\Bigr]\\
	&=\Phi^{t_0}(x,x)=\langle\cP^{t_0}x,x\rangle_{L^2(t_0,T)}
\end{align*}
for any $x\in L^2(t_0,T;\bR^d)$, where $\langle\cdot,\cdot\rangle_{L^2(t_0,T)}$ denotes the inner product on $L^2(t_0,T;\bR^d)$. We consider the following natural ansatz of the bounded self-adjoint operator $\cP^{t_0}$:
\begin{equation*}
	(\cP^{t_0}x)(t):=P^{(1)}(t)x(t)+\int^T_{t_0}P^{(2)}(t, r,t_0)x( r)\rd r,\ t\in(t_0,T),
\end{equation*}
with suitable choice of matrix-valued deterministic functions $P^{(1)}$ and $P^{(2)}$. The above ansatz is a combination of a multiplicative operator and a Hilbert--Schmidt integral operator. The purpose of this section is to derive a Lyapunov-type dynamics of the pair $(P^{(1)},P^{(2)})$.


\subsection{Preliminaries}

First, we introduce a suitable space of the pairs $(P^{(1)},P^{(2)})$ and investigate some fundamental properties.


\begin{defi}
We denote by $\Pi(0,T)$ the set of pairs $(P^{(1)},P^{(2)})$ with $P^{(1)}:(0,T)\to\bR^{d\times d}$ and $P^{(2)}:\square_3(t_0,T)\to\bR^{d\times d}$ such that
\begin{itemize}
\item
$P^{(1)}\in L^\infty(0,T;\bS^d)$;
\item
for a.e.\ $(s_1,s_2)\in(0,T)^2$, $t\mapsto P^{(2)}(s_1,s_2,t)$ is absolutely continuous on $(0,s_1\wedge s_2)$;
\item
the function $P^{(2)}(s_1,s_2,s_1\wedge s_2):=\lim_{t\uparrow s_1\wedge s_2}P(s_1,s_2,t)$, $(s_1,s_2)\in(0,T)^2$, is in $L^2((0,T)^2;\bR^{d\times d})$;
\item
the function $\dot{P}^{(2)}(s_1,s_2,t):=\frac{\partial P^{(2)}}{\partial t}(s_1,s_2,t)$, $(s_1,s_2,t)\in\square_3(0,T)$, is in $L^{2,2,1}(\square_3(0,T);\bR^{d\times d})$;
\item
for a.e.\ $(s_1,s_2,t)\in\square_3(0,T)$, it holds that $P^{(2)}(s_1,s_2,t)=P^{(2)}(s_2,s_1,t)^\top$.
\end{itemize}
\end{defi}


\begin{rem}
Observe that $\Pi(0,T)$ is a Banach space with the norm
\begin{align*}
	&\|P\|_{\Pi(0,T)}:=\underset{t\in(0,T)}{\mathrm{ess\,sup}}|P^{(1)}(t)|+\Bigl(\int^T_0\!\!\int^T_0|P^{(2)}(s_1,s_2,s_1\wedge s_2)|^2\rd s_1\!\rd s_2\Bigr)^{1/2}\\
	&\hspace{4cm}+\Bigl(\int^T_0\!\!\int^T_0\Bigl(\int^{s_1\wedge s_2}_0|\dot{P}^{(2)}(s_1,s_2,t)|\rd t\Bigr)^2\rd s_1\!\rd s_2\Bigr)^{1/2}
\end{align*}
for $P=(P^{(1)},P^{(2)})\in\Pi(0,T)$. Furthermore, the following holds:
\begin{equation*}
	\underset{t\in(0,T)}{\mathrm{ess\,sup}}|P^{(1)}(t)|+\Bigl(\int^T_0\!\!\int^T_0\sup_{t\in[0,s_1\wedge s_2]}|P^{(2)}(s_1,s_2,t)|^2\rd s_1\!\rd s_2\Bigr)^{1/2}\leq\|P\|_{\Pi(0,T)}<\infty.
\end{equation*}
\end{rem}


\begin{lemm}\label{lemm_self-adjoint}
Let $P=(P^{(1)},P^{(2)})\in\Pi(0,T)$. Then for each $t_0\in[0,T)$, the map $\cP^{t_0}:L^2(t_0,T;\bR^d)\to L^2(t_0,T;\bR^d)$ defined by
\begin{equation*}
	(\cP^{t_0}x)(t):=P^{(1)}(t)x(t)+\int^T_{t_0}P^{(2)}(t, r,t_0)x( r)\rd r,\ t\in(t_0,T),
\end{equation*}
for $x\in L^2(t_0,T;\bR^d)$, is a self-adjoint bounded linear operator on the Hilbert space $L^2(t_0,T;\bR^d)$.
\end{lemm}


\begin{proof}
It is clear that $\cP^{t_0}$ is a bounded linear operator on $L^2(t_0,T;\bR^d)$. Noting that $P^{(1)}(t)\in\bS^d$ and $P^{(2)}(s_1,s_2,t)=P^{(2)}(s_2,s_1,t)^\top$, for each $x,y\in L^2(t_0,T;\bR^d)$,
\begin{align*}
	\langle\cP^{t_0}x,y\rangle_{L^2(t_0,T)}&=\int^T_{t_0}\langle(\cP^{t_0}x)(t),y(t)\rangle\rd t\\
	&=\int^T_{t_0}\langle P^{(1)}(t)x(t),y(t)\rangle\rd t+\int^T_{t_0}\!\!\int^T_{t_0}\langle P^{(2)}(s_1,s_2,t_0)x(s_2),y(s_1)\rangle\rd s_1\!\rd s_2\\
	&=\int^T_{t_0}\langle x(t),P^{(1)}(t)y(t)\rangle\rd t+\int^T_{t_0}\!\!\int^T_{t_0}\langle x(s_2),P^{(2)}(s_2,s_1,t_0)y(s_1)\rangle\rd s_1\!\rd s_2\\
	&=\int^T_{t_0}\langle x(t),(\cP^{t_0}y)(t)\rangle\rd t=\langle x,\cP^{t_0}y\rangle_{L^2(t_0,T)}.
\end{align*}
Thus, $\cP^{t_0}$ is self-adjoint.
\end{proof}


\begin{lemm}\label{lemm_Pi_0}
Let $P=(P^{(1)},P^{(2)})\in\Pi(0,T)$. Assume that
\begin{equation*}
	\int^T_{t_0}\langle P^{(1)}(t)x(t),x(t)\rangle\rd t+\int^T_{t_0}\!\!\int^T_{t_0}\langle P^{(2)}(s_1,s_2,t_0)x(s_2),x(s_1)\rangle\rd s_1\!\rd s_2=0
\end{equation*}
for any $(t_0,x)\in\cI$. Then $P=0$ in the sense that $P^{(1)}(t)=0$ for a.e.\ $t\in(0,T)$ and $P^{(2)}(s_1,s_2,t)=0$ for any $t\in[0,s_1\wedge s_2]$ for a.e.\ $(s_1,s_2)\in(0,T)^2$.
\end{lemm}


\begin{proof}
It suffices to show that $P^{(1)}=0$. Then $P^{(2)}=0$ follows easily. Considering $t_0=0$ and $x(t):=\sqrt{N}\1_{[\tau,\tau+1/N]}(t)y$, $t\in(0,T)$, with $\tau\in(0,T)$, $N\in\bN$ with $\tau+1/N<T$ and $y\in\bR^d$ being arbitrary, we obtain
\begin{equation*}
	N\int^{\tau+1/N}_\tau\langle P^{(1)}(t)y,y\rangle\rd t+N\int^{\tau+1/N}_\tau\!\!\int^{\tau+1/N}_\tau\langle P^{(2)}(s_1,s_2,0)y,y\rangle\rd s_1\!\rd s_2=0.
\end{equation*}
Thus, we have
\begin{equation*}
	N\int^{\tau+1/N}_\tau P^{(1)}(t)\rd t+N\int^{\tau+1/N}_\tau\!\!\int^{\tau+1/N}_\tau P^{(2)}(s_1,s_2,0)\rd s_1\!\rd s_2=0.
\end{equation*}
By the Lebesgue differentiation theorem, for a.e.\ $\tau\in(0,T)$, the first term tends to $P^{(1)}(\tau)$ as $N\to\infty$. On the other hand, for any $\tau\in(0,T)$, the second term can be estimated as
\begin{equation*}
	\Bigl|N\int^{\tau+1/N}_\tau\!\!\int^{\tau+1/N}_\tau P^{(2)}(s_1,s_2,0)\rd s_1\!\rd s_2\Bigr|\leq\Bigl(\int^{\tau+1/N}_\tau\!\!\int^{\tau+1/N}_\tau|P^{(2)}(s_1,s_2,0)|^2\rd s_1\!\rd s_2\Bigr)^{1/2},
\end{equation*}
and the right-hand side tends to zero as $N\to\infty$. Thus, we have $P^{(1)}(\tau)=0$ for a.e.\ $\tau\in(0,T)$. This completes the proof.
\end{proof}

The following notations are useful for capturing the structure of some expressions arising in calculations of quadratic functionals of SVIEs.


\begin{defi}\label{defi_lint_rint}
Let $P=(P^{(1)},P^{(2)})\in\Pi(0,T)$. For each $M_1:\triangle_2(0,T)\to\bR^{d_1\times d}$ and $M_2:\triangle_2(0,T)\to\bR^{d\times d_2}$ with $d_1,d_2\in\bN$, we define
\begin{align*}
	&(M_1\lint P)(s,t):=M_1(s,t)P^{(1)}(s)+\int^T_tM_1(r,t)P^{(2)}(r,s,t)\rd r,\ (s,t)\in\triangle_2(0,T),\\
	&(P\rint M_2)(s,t):=P^{(1)}(s)M_2(s,t)+\int^T_tP^{(2)}(s,r,t)M_2(r,t)\rd r,\ (s,t)\in\triangle_2(0,T),
\end{align*}
and
\begin{equation*}
	(M_1\lint P\rint M_2)(t):=\!\int^T_tM_1(s,t)P^{(1)}(s)M_2(s,t)\rd s+\!\int^T_t\!\!\int^T_t\!M_1(s_1,t)P^{(2)}(s_1,s_2,t)M_2(s_2,t)\rd s_1\!\rd s_2,\ t\in(0,T).
\end{equation*}
\end{defi}


\begin{lemm}\label{lemm_lint-rint_integrability}
Let $P=(P^{(1)},P^{(2)})\in\Pi(0,T)$. Fix $d_1,d_2\in\bN$.
\begin{itemize}
\item[(i)]
For each $M\in \cL(\triangle_2(0,T);\bR^{d\times d_2})$ with $\cL$ being one of $L^{2,1}_\bF$, $L^2_\bF$, $L^{2,1}$, $L^2$ and $\sL^2$, it holds that $P\rint M\in \cL(\triangle_2(0,T);\bR^{d\times d_2})$ and
\begin{equation*}
	\|P\rint M\|_{\cL(\triangle_2(0,T))}\leq\|P\|_{\Pi(0,T)}\|M\|_{\cL(\triangle_2(0,T))}.
\end{equation*}
Furthermore, $(P\rint M)^\top=M^\top\lint P$.
\item[(ii)]
For each $M_1\in \sL^2(\triangle_2(0,T);\bR^{d_1\times d})$ and $M_2\in \sL^2(\triangle_2(0,T);\bR^{d\times d_2})$, it holds that $M_1\lint P\rint M_2\in L^\infty(0,T;\bR^{d_1\times d_2})$ and
\begin{equation*}
	\|M_1\lint P\rint M_2\|_{L^\infty(0,T)}\leq\|M_1\|_{\sL^2(\triangle_2(0,T))}\|P\|_{\Pi(0,T)}\|M_2\|_{\sL^2(\triangle_2(0,T))}.
\end{equation*}
Furthermore, $(M_1\lint P\rint M_2)^\top=M^\top_2\lint P\rint M^\top_1$. In particular, $M^\top_2\lint P\rint M_2\in L^\infty(0,T;\bS^{d_2})$.
\item[(iii)]
For each $M\in \sL^2(\triangle_2(0,T);\bR^{d_1\times d})$ and $\xi\in L^2_\bF(\triangle_2(0,T);\bR^{d\times d_2})$, it holds that
\begin{equation*}
	M\lint P\rint\xi\in L^2_\bF(0,T;\bR^{d_1\times d_2}),\ \xi^\top\lint P\rint\xi\in L^1_\bF(0,T;\bS^{d_2}).
\end{equation*}
\end{itemize}
\end{lemm}


\begin{proof}
This lemma can be proved by the Cauchy--Schwarz inequality and the symmetricity, and we omit the details.
\end{proof}


\subsection{Derivations of Lyapunov--Volterra equations}

By using the notations defined in the last subsection, we derive a Lyapunov-type dynamics of $P=(P^{(1)},P^{(2)})\in\Pi(0,T)$ and investigate the quadratic functional \eqref{eq_quadratic_functional_0}. Before stating the main theorem, we prove the following key lemma, which is helpful not only for understanding how to derive the Lyapunov-type equation, but also for the study of LQ stochastic Volterra control problems.


\begin{lemm}\label{lemm_Ito}
Let $P=(P^{(1)},P^{(2)})\in\Pi(0,T)$, $(t_0,x)\in\cI$ and $u\in\cU(t_0,T)$. Then
\begin{align*}
	&\bE\Bigl[\int^T_{t_0}\Bigl\{\langle P^{(1)}(t)X(t),X(t)\rangle+2\int^T_t\langle P^{(2)}(s,t,t)X(t),\Theta(s,t)\rangle\rd s\Bigr\}\rd t\Bigr]\\
	&=\int^T_{t_0}\langle P^{(1)}(t)x(t),x(t)\rangle\rd t+\int^T_{t_0}\!\!\int^T_{t_0}\langle P^{(2)}(s_1,s_2,t_0)x(s_2),x(s_1)\rangle\rd s_1\!\rd s_2+\bE\Bigl[\int^T_{t_0}(\sigma^\top\lint P\rint\sigma)(t)\rd t\Bigr]\\
	&\hspace{0.5cm}+\bE\Bigl[\int^T_{t_0}\Bigl\{\langle(D^\top\lint P\rint D)(t)u(t),u(t)\rangle\\
	&\hspace{2cm}+2\Bigl\langle(D^\top\lint P\rint C)(t)X(t)+\int^T_t(B^\top\lint P)(s,t)\Theta(s,t)\rd s+(D^\top\lint P\rint\sigma)(t),u(t)\Bigr\rangle\\
	&\hspace{2cm}+\langle(C^\top\lint P\rint C)(t)X(t),X(t)\rangle+2\int^T_t\langle(P\rint A)(s,t)X(t),\Theta(s,t)\rangle\rd s\\
	&\hspace{2cm}+\int^T_t\!\!\int^T_t\langle\dot{P}^{(2)}(s_1,s_2,t)\Theta(s_2,t),\Theta(s_1,t)\rangle\rd s_1\!\rd s_2\\
	&\hspace{2cm}+2\langle(C^\top\lint P\rint\sigma)(t),X(t)\rangle+2\int^T_t\langle(P\rint b)(s,t),\Theta(s,t)\rangle\rd s\Bigr\}\rd t\Bigr],
\end{align*}
where $X$ and $\Theta$ are the state process and the forward state process, respectively, corresponding to the input condition $(t_0,x)$ and the control $u$.
\end{lemm}


\begin{proof}
First, we note that $P^{(1)}(t)\in\bS^d$ for a.e.\ $t\in(0,T)$. Applying It\^{o}'s formula to $\langle P^{(1)}(t)\Theta(t,\cdot),\Theta(t,\cdot)\rangle$ on $(t_0,t)$ for a.e.\ $t\in(t_0,T)$, we have
\begin{align*}
	&\bE[\langle P^{(1)}(t)X(t),X(t)\rangle]=\bE[\langle P^{(1)}(t)\Theta(t,t),\Theta(t,t)\rangle]\\
	&=\langle P^{(1)}(t)x(t),x(t)\rangle\\
	&\hspace{0.5cm}+\bE\Bigl[\int^t_{t_0}\Bigl\{2\langle P^{(1)}(t)\Theta(t,s),A(t,s)X(s)+B(t,s)u(s)+b(t,s)\rangle\\
	&\hspace{2cm}+\langle P^{(1)}(t)(C(t,s)X(s)+D(t,s)u(s)+\sigma(t,s)),C(t,s)X(s)+D(t,s)u(s)+\sigma(t,s)\rangle\Bigr\}\rd s\Bigr]\\
	&=\langle P^{(1)}(t)x(t),x(t)\rangle+\bE\Bigl[\int^t_{t_0}\sigma(t,s)^\top P^{(1)}(t)\sigma(t,s)\rd s\Bigr]\\
	&\hspace{0.5cm}+\bE\Bigl[\int^t_{t_0}\Bigl\{\langle D(t,s)^\top P^{(1)}(t)D(t,s)u(s),u(s)\rangle\\
	&\hspace{2cm}+2\langle D(t,s)^\top P^{(1)}(t)C(t,s)X(s)+B(t,s)^\top P^{(1)}(t)\Theta(t,s)+D(t,s)^\top P^{(1)}(t)\sigma(t,s),u(s)\rangle\\
	&\hspace{2cm}+\langle C(t,s)^\top P^{(1)}(t)C(t,s)X(s),X(s)\rangle+2\langle P^{(1)}(t)A(t,s)X(s),\Theta(t,s)\rangle\\
	&\hspace{2cm}+2\langle C(t,s)^\top P^{(1)}(t)\sigma(t,s),X(s)\rangle+2\langle P^{(1)}(t)b(t,s),\Theta(t,s)\rangle\Bigr\}\rd s\Bigr].
\end{align*}
Integrating both sides with respect to $t\in(t_0,T)$ and applying Fubini's theorem, we get
\begin{equation}\label{eq_XX}
\begin{split}
	&\bE\Bigl[\int^T_{t_0}\langle P^{(1)}(t)X(t),X(t)\rangle\rd t\Bigr]\\
	&=\int^T_{t_0}\langle P^{(1)}(t)x(t),x(t)\rangle\rd t+\bE\Bigl[\int^T_{t_0}\!\!\int^T_t\sigma(s,t)^\top P^{(1)}(s)\sigma(s,t)\rd s\!\rd t\Bigr]\\
	&\hspace{0.5cm}+\bE\Bigl[\int^T_{t_0}\Bigl\{\Bigl\langle\int^T_tD(s,t)^\top P^{(1)}(s)D(s,t)\rd s\,u(t),u(t)\Bigr\rangle\\
	&\hspace{2cm}+2\Bigl\langle\int^T_tD(s,t)^\top P^{(1)}(s)C(s,t)\rd s\,X(t)+\int^T_tB(s,t)^\top P^{(1)}(s)\Theta(s,t)\rd s\\
	&\hspace{4cm}+\int^T_tD(s,t)^\top P^{(1)}(s)\sigma(s,t)\rd s,u(t)\Bigr\rangle\\
	&\hspace{2cm}+\Bigl\langle\int^T_tC(s,t)^\top P^{(1)}(s)C(s,t)\rd s\,X(t),X(t)\Bigr\rangle+2\int^T_t\langle P^{(1)}(s)A(s,t)X(t),\Theta(s,t)\rangle\rd s\\
	&\hspace{2cm}+2\Bigl\langle\int^T_tC(s,t)^\top P^{(1)}(s)\sigma(s,t)\rd s,X(t)\Bigr\rangle+2\int^T_t\langle P^{(1)}(s)b(s,t),\Theta(s,t)\rangle\rd s\Bigr\}\rd t\Bigr].
\end{split}
\end{equation}
Next, we note that $P^{(2)}(s_1,s_2,t)=P^{(2)}(s_2,s_1,t)^\top$ for any $t\in[0,s_1\wedge s_2]$ for a.e.\ $(s_1,s_2)\in(0,T)^2$. Applying It\^{o}'s formula to $\langle P^{(2)}(s,t,\cdot)\Theta(t,\cdot),\Theta(s,\cdot)\rangle$ on $(t_0,t)$ for a.e.\ $(s,t)\in\triangle_2(t_0,T)$, we have
\begin{align*}
	&2\bE[\langle P^{(2)}(s,t,t)X(t),\Theta(s,t)\rangle]=2\bE[\langle P^{(2)}(s,t,t)\Theta(t,t),\Theta(s,t)\rangle]\\
	&=2\langle P^{(2)}(s,t,t_0)x(t),x(s)\rangle\\
	&\hspace{0.5cm}+2\bE\Bigl[\int^t_{t_0}\Bigl\{\langle \dot{P}^{(2)}(s,t,r)\Theta(t,r),\Theta(s,r)\rangle\\
	&\hspace{2cm}+\langle P^{(2)}(s,t,r)\Theta(t,r),A(s,r)X(r)+B(s,r)u(r)+b(s,r)\rangle\\
	&\hspace{2cm}+\langle P^{(2)}(s,t,r)(A(t,r)X(r)+B(t,r)u(r)+b(t,r)),\Theta(s,r)\rangle\\
	&\hspace{2cm}+\langle P^{(2)}(s,t,r)(C(t,r)X(r)+D(t,r)u(r)+\sigma(t,r)),C(s,r)X(r)+D(s,r)u(r)+\sigma(s,r)\rangle\Bigr\}\rd r\Bigr]\\
	&=\langle P^{(2)}(s,t,t_0)x(t),x(s)\rangle+\langle P^{(2)}(t,s,t_0)x(s),x(t)\rangle\\
	&\hspace{0.5cm}+\bE\Bigl[\int^t_{t_0}\{\sigma(s,r)^\top P^{(2)}(s,t,r)\sigma(t,r)+\sigma(t,r)^\top P^{(2)}(t,s,r)\sigma(s,r)\}\rd r\Bigr]\\
	&\hspace{0.5cm}+\bE\Bigl[\int^t_{t_0}\Bigl\{\langle(D(s,r)^\top P^{(2)}(s,t,r)D(t,r)+D(t,r)^\top P^{(2)}(t,s,r)D(s,r))u(r),u(r)\rangle\\
	&\hspace{2cm}+2\langle(D(s,r)^\top P^{(2)}(s,t,r)C(t,r)+D(t,r)^\top P^{(2)}(t,s,r)C(s,r))X(r)\\
	&\hspace{4cm}+B(s,r)^\top P^{(2)}(s,t,r)\Theta(t,r)+B(t,r)^\top P^{(2)}(t,s,r)\Theta(s,r)\\
	&\hspace{4cm}+D(s,r)^\top P^{(2)}(s,t,r)\sigma(t,r)+D(t,r)^\top P^{(2)}(t,s,r)\sigma(s,r),u(r)\rangle\\
	&\hspace{2cm}+\langle(C(s,r)^\top P^{(2)}(s,t,r)C(t,r)+C(t,r)^\top P^{(2)}(t,s,r)C(s,r))X(r),X(r)\rangle\\
	&\hspace{2cm}+2\langle P^{(2)}(s,t,r)A(t,r)X(r),\Theta(s,r)\rangle+2\langle P^{(2)}(t,s,r)A(s,r)X(r),\Theta(t,r)\rangle\\
	&\hspace{2cm}+\langle\dot{P}^{(2)}(s,t,r)\Theta(t,r),\Theta(s,r)\rangle+\langle\dot{P}^{(2)}(t,s,r)\Theta(s,r),\Theta(t,r)\rangle\\
	&\hspace{2cm}+2\langle C(s,r)^\top P^{(2)}(s,t,r)\sigma(t,r)+C(t,r)^\top P^{(2)}(t,s,r)\sigma(s,r),X(r)\rangle\\
	&\hspace{2cm}+2\langle P^{(2)}(s,t,r)b(t,r),\Theta(s,r)\rangle+2\langle P^{(2)}(t,s,r)b(s,r),\Theta(t,r)\rangle\Bigr\}\rd r\Bigr].
\end{align*}
We integrate both sides with respect to $(s,t)\in\triangle_2(t_0,T)$ and apply Fubini's theorem. Noting that
\begin{equation*}
	\int^T_{t_0}\!\!\int^T_t\{f(s,t)+f(t,s)\}\rd s\!\rd t=\int_{\triangle_2(t_0,T)}\{f(s,t)+f(t,s)\}\rd(s,t)=\int^T_{t_0}\!\!\int^T_{t_0}f(s_1,s_2)\rd s_1\!\rd s_2
\end{equation*}
for each function $f$ on $(t_0,T)^2$, we obtain
\begin{equation}\label{eq_ThetaX}
\begin{split}
	&2\bE\Bigl[\int^T_{t_0}\!\!\int^T_t\langle P^{(2)}(s,t,t)X(t),\Theta(s,t)\rangle\rd s\!\rd t\Bigr]\\
	&=\int^T_{t_0}\!\!\int^T_{t_0}\langle P^{(2)}(s_1,s_2,t_0)x(s_2),x(s_1)\rangle\rd s_1\!\rd s_2+\bE\Bigl[\int^T_{t_0}\int^T_t\!\!\int^T_t\sigma(s_1,t)^\top P^{(2)}(s_1,s_2,t)\sigma(s_2,t)\rd s_1\!\rd s_2\rd t\Bigr]\\
	&\hspace{0.5cm}+\bE\Bigl[\int^T_{t_0}\Bigl\{\Bigl\langle\int^T_t\!\!\int^T_tD(s_1,t)^\top P^{(2)}(s_1,s_2,t)D(s_2,t)\rd s_1\!\rd s_2\,u(t),u(t)\Bigr\rangle\\
	&\hspace{2cm}+2\Bigl\langle\int^T_t\!\!\int^T_tD(s_1,t)^\top P^{(2)}(s_1,s_2,t)C(s_2,t)\rd s_1\!\rd s_2\,X(t)\\
	&\hspace{4cm}+\int^T_t\!\!\int^T_tB(r,t)^\top P^{(2)}(r,s,t)\rd r\,\Theta(s,t)\rd s\\
	&\hspace{4cm}+\int^T_t\!\!\int^T_tD(s_1,t)^\top P^{(2)}(s_1,s_2,t)\sigma(s_2,t)\rd s_1\!\rd s_2,u(t)\Bigr\rangle\\
	&\hspace{2cm}+\Bigl\langle\int^T_t\!\!\int^T_tC(s_1,t)^\top P^{(2)}(s_1,s_2,t)C(s_2,t)\rd s_1\!\rd s_2\,X(t),X(t)\Bigr\rangle\\
	&\hspace{2cm}+2\int^T_t\Bigl\langle\int^T_t P^{(2)}(s,r,t)A(r,t)\rd r\,X(t),\Theta(s,t)\Bigr\rangle\rd s\\
	&\hspace{2cm}+\int^T_t\!\!\int^T_t\langle\dot{P}^{(2)}(s_1,s_2,t)\Theta(s_2,t),\Theta(s_1,t)\rangle\rd s_1\!\rd s_2\\
	&\hspace{2cm}+2\Bigl\langle\int^T_t\!\!\int^T_tC(s_1,t)^\top P^{(2)}(s_1,s_2,t)\sigma(s_2,t)\rd s_1\!\rd s_2,X(t)\Bigr\rangle\\
	&\hspace{2cm}+2\int^T_t\Bigl\langle\int^T_tP^{(2)}(s,r,t)b( r,t)\rd r,\Theta(s,t)\Bigr\rangle\rd s\Bigr\}\rd t\Bigr].
\end{split}
\end{equation}
By \eqref{eq_XX} and \eqref{eq_ThetaX}, noting \cref{defi_lint_rint}, we get the assertion.
\end{proof}


\begin{rem}
A key point in the above proof is the use of It\^{o}'s formula to some quadratic form of $\Theta(\cdot,\cdot)$, from which one obtains the corresponding quadratic form of $X(\cdot)$. This kind of techniques also appeared in \cite{WaT18} where LQ stochastic Volterra control problems were treated in the open-loop framework.
\end{rem}


\begin{cor}\label{cor_Ito_feedback}
Let $P=(P^{(1)},P^{(2)})\in\Pi(0,T)$ and $(\Xi,\Gamma)\in L^\infty(0,T;\bR^{\ell\times d})\times L^2(\triangle_2(0,T);\bR^{\ell\times d})$ be given. For each $(t_0,x)\in\cI$, let $(X^{t_0,x}_0,\Theta^{t_0,x}_0)\in L^2_\bF(t_0,T;\bR^d)\times L^2_{\bF,\mathrm{c}}(\triangle_2(t_0,T);\bR^d)$ be the causal feedback solution to the homogeneous controlled SVIE \eqref{eq_state_0} at $(t_0,x)\in\cI$ corresponding to the causal feedback strategy $(\Xi,\Gamma,0)\in\cS(0,T)$. Then it holds that
\begin{align*}
	&\int^T_{t_0}\langle P^{(1)}(t)x(t),x(t)\rangle\rd t+\int^T_{t_0}\!\!\int^T_{t_0}\langle P^{(2)}(s_1,s_2,t_0)x(s_2),x(s_1)\rangle\rd s_1\!\rd s_2\\
	&=\bE\Bigl[\int^T_{t_0}\Bigl\{\langle (P^{(1)}(t)-F^{(1)}[\Xi;P](t))X^{t_0,x}_0(t),X^{t_0,x}_0(t)\rangle\\
	&\hspace{2cm}+2\int^T_t\langle(P^{(2)}(s,t,t)-F^{(2)}[\Xi,\Gamma;P](s,t))X^{t_0,x}_0(t),\Theta^{t_0,x}_0(s,t)\rangle\rd s\\
	&\hspace{2cm}-\int^T_t\!\!\int^T_t\langle(\dot{P}^{(2)}(s_1,s_2,t)+F^{(3)}[\Gamma;P](s_1,s_2,t))\Theta^{t_0,x}_0(s_2,t),\Theta^{t_0,x}_0(s_1,t)\rangle\rd s_1\!\rd s_2\Bigr\}\rd t\Bigr],
\end{align*}
where
\begin{equation}\label{eq_Lyapunov--Volterra_coefficients}
\begin{split}
	&F^{(1)}[\Xi;P](t):=(C^\top\lint P\rint C)(t)+\Xi(t)^\top(D^\top\lint P\rint C)(t)+(C^\top\lint P\rint D)(t)\Xi(t)+\Xi(t)^\top(D^\top\lint P\rint D)(t)\Xi(t),\\
	&\hspace{6cm}t\in(0,T),\\
	&F^{(2)}[\Xi,\Gamma;P](s,t):=(P\rint A)(s,t)+(P\rint B)(s,t)\Xi(t)+\Gamma(s,t)^\top(D^\top\lint P\rint C)(t)+\Gamma(s,t)^\top(D^\top\lint P\rint D)(t)\Xi(t),\\
	&\hspace{6cm}(s,t)\in\triangle_2(0,T),\\
	&F^{(3)}[\Gamma;P](s_1,s_2,t):=\Gamma(s_1,t)^\top(B^\top\lint P)(s_2,t)+(P\rint B)(s_1,t)\Gamma(s_2,t)+\Gamma(s_1,t)^\top(D^\top\lint P\rint D)(t)\Gamma(s_2,t),\\
	&\hspace{6cm}(s_1,s_2,t)\in\square_3(0,T).
\end{split}
\end{equation}
\end{cor}


\begin{proof}
Denote the outcome by $u^{t_0,x}_0:=(\Xi,\Gamma,0)^0[t_0,x]$. By using \cref{lemm_Ito} with $(b,\sigma)=(0,0)$, we have
\begin{align*}
	&\int^T_{t_0}\langle P^{(1)}(t)x(t),x(t)\rangle\rd t+\int^T_{t_0}\!\!\int^T_{t_0}\langle P^{(2)}(s_1,s_2,t_0)x(s_2),x(s_1)\rangle\rd s_1\!\rd s_2\\
	&=\bE\Bigl[\int^T_{t_0}\Bigl\{\langle P^{(1)}(t)X^{t_0,x}_0(t),X^{t_0,x}_0(t)\rangle+2\int^T_t\langle P^{(2)}(s,t,t)X^{t_0,x}_0(t),\Theta^{t_0,x}_0(s,t)\rangle\rd s\Bigr\}\rd t\Bigr]\\
	&\hspace{0.5cm}-\bE\Bigl[\int^T_{t_0}\Bigl\{\langle(D^\top\lint P\rint D)(t)u^{t_0,x}_0(t),u^{t_0,x}_0(t)\rangle\\
	&\hspace{2cm}+2\Bigl\langle(D^\top\lint P\rint C)(t)X^{t_0,x}_0(t)+\int^T_t(B^\top\lint P)(s,t)\Theta^{t_0,x}_0(s,t)\rd s,u^{t_0,x}_0(t)\Bigr\rangle\\
	&\hspace{2cm}+\langle(C^\top\lint P\rint C)(t)X^{t_0,x}_0(t),X^{t_0,x}_0(t)\rangle+2\int^T_t\langle (P\rint A)(s,t)X^{t_0,x}_0(t),\Theta^{t_0,x}_0(s,t)\rangle\rd s\\
	&\hspace{2cm}+\int^T_t\!\!\int^T_t\langle\dot{P}^{(2)}(s_1,s_2,t)\Theta^{t_0,x}_0(s_2,t),\Theta^{t_0,x}_0(s_1,t)\rangle\rd s_1\!\rd s_2\Bigr\}\rd t\Bigr].
\end{align*}
By using the relation $u^{t_0,x}_0(t)=\Xi(t)X^{t_0,x}_0(t)+\int^T_t\Gamma(s,t)\Theta^{t_0,x}_0(s,t)\rd s$ and noting the definition \eqref{eq_Lyapunov--Volterra_coefficients} of $(F^{(1)},F^{(2)},F^{(3)})$, we get the assertion.
\end{proof}


\begin{rem}
By the definition \eqref{eq_Lyapunov--Volterra_coefficients}, together with \cref{lemm_lint-rint_integrability}, we see that the map
\begin{equation*}
	P\mapsto(F^{(1)}[\Xi;P],F^{(2)}[\Xi,\Gamma;P],F^{(3)}[\Gamma;P]).
\end{equation*}
is a bounded linear operator from $\Pi(0,T)$ to $L^\infty(0,T;\bS^d)\times L^2(\triangle_2(0,T);\bR^{d\times d})\times L^{2,2,1}_\sym(\square_3(0,T);\bR^{d\times d})$.
\end{rem}

Now we derive a dynamics for $P=(P^{(1)},P^{(2)})\in\Pi(0,T)$ which represents the quadratic functional \eqref{eq_quadratic_functional_0}. Let $(Q^{(1)},Q^{(2)},Q^{(3)})\in L^\infty(0,T;\bS^d)\times L^2(\triangle_2(0,T);\bR^{d\times d})\times L^{2,2,1}_\sym(\square_3(0,T);\bR^{d\times d})$ be given.
From \cref{cor_Ito_feedback}, it is natural to introduce the following equation:
\begin{equation}\label{eq_Lyapunov--Volterra}
	\begin{dcases}
	P^{(1)}(t)=F^{(1)}[\Xi;P](t)+Q^{(1)}(t),\ t\in(0,T),\\
	P^{(2)}(s,t,t)=P^{(2)}(t,s,t)^\top=F^{(2)}[\Xi,\Gamma;P](s,t)+Q^{(2)}(s,t),\ (s,t)\in\triangle_2(0,T),\\
	\dot{P}^{(2)}(s_1,s_2,t)+F^{(3)}[\Gamma;P](s_1,s_2,t)+Q^{(3)}(s_1,s_2,t)=0,\ (s_1,s_2,t)\in\square_3(0,T),
	\end{dcases}
\end{equation}
with $(F^{(1)}[\Xi;P],F^{(2)}[\Xi,\Gamma;P],F^{(3)}[\Gamma;P])$ defined by \eqref{eq_Lyapunov--Volterra_coefficients}. In the integral notation, \eqref{eq_Lyapunov--Volterra} is rewritten as
\begin{equation*}
	\begin{dcases}
	P^{(1)}(t)=\int^T_t(C(s,t)+D(s,t)\Xi(t))^\top P^{(1)}(s)(C(s,t)+D(s,t)\Xi(t))\rd s\\
	\hspace{2cm}+\int^T_t\!\!\int^T_t(C(s_1,t)+D(s_1,t)\Xi(t))^\top P^{(2)}(s_1,s_2,t)(C(s_2,t)+D(s_2,t)\Xi(t))\rd s_1\!\rd s_2\\
	\hspace{2cm}+Q^{(1)}(t),\ t\in(0,T),\\
	P^{(2)}(s,t,t)=P^{(2)}(t,s,t)^\top\\
	=P^{(1)}(s)(A(s,t)+B(s,t)\Xi(t))+\Gamma(s,t)^\top\int^T_tD(r,t)^\top P^{(1)}(r)(C(r,t)+D(r,t)\Xi(t))\rd r\\
	\hspace{0.5cm}+\int^T_tP^{(2)}(s,r,t)(A(r,t)+B(r,t)\Xi(t))\rd r\\
	\hspace{0.5cm}+\Gamma(s,t)^\top\int^T_t\!\!\int^T_tD(r_1,t)^\top P^{(2)}(r_1,r_2,t)(C(r_2,t)+D(r_2,t)\Xi(t))\rd r_1\!\rd r_2\\
	\hspace{0.5cm}+Q^{(2)}(s,t),\ (s,t)\in\triangle_2(0,T),\\
	\dot{P}^{(2)}(s_1,s_2,t)+\Gamma(s_1,t)^\top B(s_2,t)^\top P^{(1)}(s_2)+P^{(1)}(s_1)B(s_1,t)\Gamma(s_2,t)\\
	\hspace{0.2cm}+\Gamma(s_1,t)^\top\int^T_tD(r,t)^\top P^{(1)}(r)D(r,t)\rd r\,\Gamma(s_2,t)\\
	\hspace{0.2cm}+\Gamma(s_1,t)^\top\int^T_tB(r,t)^\top P^{(2)}(r,s_2,t)\rd r+\int^T_tP^{(2)}(s_1,r,t)B(r,t)\rd r\,\Gamma(s_2,t)\\
	\hspace{0.2cm}+\Gamma(s_1,t)^\top\int^T_t\!\!\int^T_tD(r_1,t)^\top P^{(2)}(r_1,r_2,t)D(r_2,t)\rd r_1\!\rd r_2\,\Gamma(s_2,t)\\
	\hspace{0.2cm}+Q^{(3)}(s_1,s_2,t)=0,\ (s_1,s_2,t)\in\square_3(0,T).
	\end{dcases}
\end{equation*}
The above is a coupled system of Lyapunov-type (backward) Volterra integro-differential equations for the pair $P=(P^{(1)},P^{(2)})\in\Pi(0,T)$ of matrix-valued deterministic functions. In this paper, we call the above equation a \emph{Lyapunov--Volterra equation}. We say that $P=(P^{(1)},P^{(2)})\in\Pi(0,T)$ satisfying the above equalities a solution to the Lyapunov--Volterra equation \eqref{eq_Lyapunov--Volterra}.

Now we are ready to state the third main result of this paper.


\begin{theo}\label{theo_Lyapunov--Volterra}
Let $(\Xi,\Gamma)\in L^\infty(0,T;\bR^{\ell\times d})\times L^2(\triangle_2(0,T);\bR^{\ell\times d})$. For each $(Q^{(1)},Q^{(2)},Q^{(3)})\in L^\infty(0,T;\bS^d)\times L^2(\triangle_2(0,T);\bR^{d\times d})\times L^{2,2,1}_\sym(\square_3(0,T);\bR^{d\times d})$, there exists a unique solution $P=(P^{(1)},P^{(2)})\in\Pi(0,T)$ to the Lyapunov--Volterra equation \eqref{eq_Lyapunov--Volterra}. Furthermore, for any $(t_0,x)\in\cI$, the following holds:
\begin{equation}\label{eq_quadratic_representation}
\begin{split}
	&\bE\Bigl[\int^T_{t_0}\Bigl\{\langle Q^{(1)}(t)X^{t_0,x}_0(t),X^{t_0,x}_0(t)\rangle+2\int^T_t\langle Q^{(2)}(s,t)X^{t_0,x}_0(t),\Theta^{t_0,x}_0(s,t)\rangle\rd s\\
	&\hspace{2cm}+\int^T_t\!\!\int^T_t\langle Q^{(3)}(s_1,s_2,t)\Theta^{t_0,x}_0(s_2,t),\Theta^{t_0,x}_0(s_1,t)\rangle\rd s_1\!\rd s_2\Bigr\}\rd t\Bigr]\\
	&=\int^T_{t_0}\langle P^{(1)}(t)x(t),x(t)\rangle\rd t+\int^T_{t_0}\!\!\int^T_{t_0}\langle P^{(2)}(s_1,s_2,t_0)x(s_2),x(s_1)\rangle\rd s_1\!\rd s_2,
\end{split}
\end{equation}
where $(X^{t_0,x}_0,\Theta^{t_0,x}_0)$ is the causal feedback solution to the homogeneous controlled SVIE \eqref{eq_state_0} at $(t_0,x)$ corresponding to $(\Xi,\Gamma,0)\in\cS(0,T)$.
\end{theo}


\begin{proof}
See \cref{appendix_Lyapunov--Volterra}.
\end{proof}


\appendix

\section{Appendix: Proofs of the main results}\label{appendix}

In this appendix, we prove our main Theorems \ref{theo_SVIE}, \ref{theo_Type-II_EBSVIE} and \ref{theo_Lyapunov--Volterra}. The following lemma is an abstract formulation of the so-called \emph{method of continuation} for (non-linear) equations on a Banach space. This is crucial for the proof of the well-posedness of several kinds of equations arising in this paper.


\begin{lemm}\label{lemm_method_of_continuation}
Let $(\cX,\|\cdot\|_\cX)$ be a Banach space, $(\cY,\|\cdot\|_\cY)$ be a normed space, and $\cT:\cY\to\cX$ be a continuous map. Suppose that a map $F:\cX\to\cY$ satisfies the following two properties:
\begin{itemize}
\item
Lipschitz continuity: There exists a constant $L>0$ such that
\begin{equation*}
	\|F(X_1)-F(X_2)\|_\cY\leq L\|X_1-X_2\|_\cX
\end{equation*}
for any $X_1,X_2\in\cX$.
\item
A priori estimate: There exists a constant $K>0$ such that, for any $\varphi_1,\varphi_2\in\cY$ and $\lambda\in[0,1]$, two solutions $X_1,X_2\in\cX$ of the equations
\begin{equation*}
	X_i=\cT(\lambda F(X_i)+\varphi_i),\ i=1,2,
\end{equation*}
satisfy
\begin{equation*}
	\|X_1-X_2\|_\cX\leq K\|\varphi_1-\varphi_2\|_\cY.
\end{equation*}
\end{itemize}
Then, for any $\varphi\in\cY$, there exists a unique solution $X\in\cX$ to the equation
\begin{equation*}
	X=\cT(F(X)+\varphi).
\end{equation*}
\end{lemm}


\begin{proof}
The uniqueness of the solution follows from the a priori estimate. We prove the existence by \emph{the method of continuation}. Let $N\in\bN$ be a number such that $N\geq2(1+KL)$. For each $n\in\{0,1,\dots,N\}$, we say that (P$_n$) holds if for any $\varphi\in\cY$, there exists a unique solution $X\in\cX$ to the equation $X=\cT(\frac{n}{N}F(X)+\varphi)$. We shall show that (P$_N$) holds. Clearly, (P$_0$) holds. Assume that (P$_n$) holds for some $n\in\{0,1,\dots,N-1\}$. Fix an arbitrary $\varphi\in\cY$. By virtue of the assumption (P$_n$), we can define a sequence $X_i\in\cX$ with $i\in\bN$ inductively as the solution to the equation
\begin{equation*}
	X_i=\cT\Bigl(\frac{n}{N}F(X_i)+\frac{1}{N}F(X_{i-1})+\varphi\Bigr),
\end{equation*}
with the convention that $X_0:=0$. By the a priori estimate and the Lipschitz continuity of $F$, for each $i\in\bN$, we have
\begin{align*}
	\|X_{i+1}-X_i\|_\cX\leq\frac{K}{N}\|F(X_i)-F(X_{i-1})\|_\cY\leq\frac{KL}{N}\|X_i-X_{i-1}\|_\cX\leq\frac{1}{2}\|X_i-X_{i-1}\|_\cX,
\end{align*}
where the last inequality follows from $N\geq2(1+KL)$. The above estimate implies that $\{X_i\}_{i\in\bN}$ is a Cauchy sequence on the Banach space $\cX$, and thus the limit $X:=\lim_{i\to\infty}X_i$ exists. Noting the continuity of $F$ and $\cT$, we obtain
\begin{equation*}
	X=\cT\Bigl(\frac{n}{N}F(X)+\frac{1}{N}F(X)+\varphi\Bigr)=\cT\Bigl(\frac{n+1}{N}F(X)+\varphi\Bigr).
\end{equation*}
Since $\varphi\in\cY$ is arbitrary, we see that (P$_{n+1}$) holds. By the induction, we see that (P$_N$) holds. This completes the proof.
\end{proof}


\begin{rem}
We will use the above abstract lemma to show the existence and uniqueness of the solutions to SVIEs, Type-II EBSVIEs and Lyapunov--Volterra equations. In the framework of this paper, we typically consider the case where $\cT$ and $F$ are bounded linear operators. In this case, the Lipschitz continuity of $F$ is trivial, and the a priori estimate of the above form is a key step to show the well-posedness of the equations.
\end{rem}


\subsection{Proof of \cref{theo_SVIE}: Causal feedback solutions of controlled SVIEs}\label{appendix_SVIE}

First, we prove the well-posedness of a linear (uncontrolled) SVIE. Although the following lemma follows from the general theory of \cite{Ha21+++}, we give a simple proof here by means of \cref{lemm_method_of_continuation} for the sake of self-containedness.


\begin{lemm}\label{lemm_SVIE_general}
Let $A\in L^2(\triangle_2(0,T);\bR^{d\times d})$ and $C\in\sL^2(\triangle_2(0,T);\bR^{d\times d})$ be given. For any $\varphi\in L^2_\bF(t_0,T;\bR^d)$ with $t_0\in[0,T)$, there exists a unique solution $X\in L^2_\bF(t_0,T;\bR^d)$ to the SVIE
\begin{equation*}
	X(t)=\varphi(t)+\int^t_{t_0}A(t,s)X(s)\rd s+\int^t_{t_0}C(t,s)X(s)\rd W(s),\ t\in(t_0,T).
\end{equation*}
Furthermore, the following estimate holds:
\begin{equation*}
	\|X\|_{L^2_\bF(t_0,T)}\leq2m^2(1+2L)^{m-1}\|\varphi\|_{L^2_\bF(t_0,T)},
\end{equation*}
where $L:=\|A\|_{L^2(\triangle_2(0,T))}+\|C\|_{\sL^2(\triangle_2(0,T))}$, and $m\in\bN$ is the number of division intervals $0=U_0<U_1<\cdots<U_m=T$ of $(0,T)$ satisfying
\begin{equation*}
	\Bigl(\int^{U_{i+1}}_{U_i}\!\!\int^t_{U_i}|A(t,s)|^2\rd s\!\rd t\Bigr)^{1/2}+\underset{t\in(U_i,U_{i+1})}{\mathrm{ess\,sup}}\Bigl(\int^{U_{i+1}}_t|C(s,t)|^2\rd s\Bigr)^{1/2}\leq\frac{1}{2}
\end{equation*}
for $i=0,1,\dots,m-1$.
\end{lemm}


\begin{proof}
It suffices to show the theorem for $t_0=0$. Define $F:L^2_\bF(0,T;\bR^d)\to L^2_\bF(0,T;\bR^d)$ by
\begin{equation*}
	F(X)(t):=\int^t_0A(t,s)X(s)\rd s+\int^t_0C(t,s)X(s)\rd W(s),\ t\in(0,T),
\end{equation*}
for $X\in L^2_\bF(0,T;\bR^d)$. We prove that, for any $\varphi\in L^2_\bF(0,T;\bR^d)$, the equation $X=F(X)+\varphi$ has a unique solution $X\in L^2_\bF(0,T;\bR^d)$. It is easy to see that $F$ is a bounded linear operator on $L^2_\bF(0,T;\bR^d)$. We prove the a priori estimate. Let $L:=\|A\|_{L^2(\triangle_2(0,T))}+\|C\|_{\sL^2(\triangle_2(0,T))}$ and $0=U_0<U_1<\cdots<U_m=T$ be a partition of $(0,T)$ satisfying the requirement in the assertion. We emphasize that such a partition exists since $A\in L^2(\triangle_2(0,T);\bR^{d\times d})$ and $C\in\sL^2(\triangle_2(0,T);\bR^{d\times d})$. Let $\varphi_1,\varphi_2\in L^2_\bF(0,T;\bR^d)$ and $\lambda\in[0,1]$ be fixed, and assume that $X_1,X_2\in L^2_\bF(0,T;\bR^d)$ satisfy $X_i=\lambda F(X_i)+\varphi_i$ for $i=1,2$. Define $\bar{X}:=X_1-X_2$ and $\bar{\varphi}:=\varphi_1-\varphi_2$. Then we have $\bar{X}=\lambda F(\bar{X})+\bar{\varphi}$. For $i\in\{0,1,\dots,m-1\}$, we have
\begin{equation*}
	\bar{X}(t)=\bar{\varphi}(t)+\lambda\{F^1_i(\bar{X})(t)+F^2_i(\bar{X})(t)\},\ t\in(U_i,U_{i+1}),
\end{equation*}
where
\begin{align*}
	&F^1_i(\bar{X})(t):=\int^{U_i}_0A(t,s)\bar{X}(s)\rd s+\int^{U_i}_0C(t,s)\bar{X}(s)\rd W(s),\ t\in(U_i,U_{i+1}),\\
	&F^2_i(\bar{X})(t):=\int^t_{U_i}A(t,s)\bar{X}(s)\rd s+\int^t_{U_i}C(t,s)\bar{X}(s)\rd W(s),\ t\in(U_i,U_{i+1}).
\end{align*}
It is easy to see that
\begin{equation*}
	\|F^1_i(\bar{X})\|_{L^2_\bF(U_i,U_{i+1})}\leq L\|\bar{X}\|_{L^2_\bF(0,U_i)},\ \|F^2_i(\bar{X})\|_{L^2_\bF(U_i,U_{i+1})}\leq\frac{1}{2}\|\bar{X}\|_{L^2_\bF(U_i,U_{i+1})}.
\end{equation*}
Thus, we have
\begin{equation*}
	\|\bar{X}\|_{L^2_\bF(U_i,U_{i+1})}\leq 2\|\bar{\varphi}\|_{L^2_\bF(U_i,U_{i+1})}+2L\|\bar{X}\|_{L^2_\bF(0,U_i)}\leq 2\|\bar{\varphi}\|_{L^2_\bF(0,T)}+2L\sum^{i-1}_{j=0}\|\bar{X}\|_{L^2_\bF(U_j,U_{j+1})}.
\end{equation*}
By the discrete Gronwall inequality, we have
\begin{equation*}
	\max_{i\in\{0,1,\dots,m-1\}}\|\bar{X}\|_{L^2_\bF(U_i,U_{i+1})}\leq2\|\bar{\varphi}\|_{L^2_\bF(0,T)}\sum^m_{j=1}(1+2L)^{j-1}\leq 2m(1+2L)^{m-1}\|\bar{\varphi}\|_{L^2_\bF(0,T)},
\end{equation*}
and thus
\begin{equation*}
	\|\bar{X}\|_{L^2_\bF(0,T)}\leq\sum^{m-1}_{i=0}\|\bar{X}\|_{L^2_\bF(U_i,U_{i+1})}\leq2m^2(1+2L)^{m-1}\|\bar{\varphi}\|_{L^2_\bF(0,T)}.
\end{equation*}
Therefore, the a priori estimate holds. By \cref{lemm_method_of_continuation}, for any $\varphi\in L^2_\bF(0,T;\bR^d)$, the SVIE $X=F(X)+\varphi$ has a unique solution $X\in L^2_\bF(0,T;\bR^d)$.
\end{proof}

The following technical lemma reveals the structure of the causal feedback solution of a linear controlled SVIE.


\begin{lemm}\label{lemm_closed-loop_transform}
Let $(\Xi,\Gamma,v)\in\cS(0,T)$ be fixed. For each $(t_0,x)\in\cI$, a triplet $(\Xi^{t_0,x},\Theta^{t_0,x},u^{t_0,x})\in L^2_\bF(t_0,T;\bR^d)\times L^2_{\bF,\mathrm{c}}(\triangle_2(t_0,T);\bR^d)\times\cU(t_0,T)$ is a causal feedback solution to the controlled SVIE \eqref{eq_state} if and only if the following holds:
\begin{equation}\label{eq_closed-loop_transform}
	\begin{dcases}
	\begin{pmatrix}X^{t_0,x}(t)\\u^{t_0,x}(t)\end{pmatrix}=\bX^{t_0,x}(t),\ t\in(t_0,T),\\
	\Theta^{t_0,x}(s,t)=x(s)+\int^t_{t_0}\{(A(s,r),B(s,r))\bX^{t_0,x}(r)+b(s,r)\}\rd r\\
	\hspace{3cm}+\int^t_{t_0}\{(C(s,r),D(s,r))\bX^{t_0,x}(r)+\sigma(s,r)\}\rd W(r),\ (s,t)\in\triangle_2(t_0,T),
	\end{dcases}
\end{equation}
where $\bX^{t_0,x}\in L^2_\bF(t_0,T;\bR^{d+\ell})$ is the solution to the SVIE
\begin{equation}\label{eq_closed-loop_system'}
	\bX^{t_0,x}(t)=\Phi(t)+\int^t_{t_0}\bA(t,s)\bX^{t_0,x}(s)\rd s+\int^t_{t_0}\bC(t,s)\bX^{t_0,x}(s)\rd W(s),\ t\in(t_0,T).
\end{equation}
Here, $\Phi\in L^2_\bF(t_0,T;\bR^{d+\ell})$, $\bA\in L^2(\triangle_2(0,T);\bR^{(d+\ell)\times(d+\ell)})$ and $\bC\in\sL^2(\triangle_2(0,T);\bR^{(d+\ell)\times(d+\ell)})$ are defined by
\begin{align*}
	&\Phi(t):=\begin{pmatrix}x(t)+\int^t_{t_0}b(t,s)\rd s+\int^t_{t_0}\sigma(t,s)\rd W(s)\\x^{\Xi,\Gamma}(t)+\int^t_{t_0}b^{\Xi,\Gamma}(t,s)\rd s+\int^t_{t_0}\sigma^{\Xi,\Gamma}(t,s)\rd W(s)+v(t)\end{pmatrix},\ t\in(t_0,T),\\
	&\bA(t,s):=\begin{pmatrix}A(t,s)&B(t,s)\\A^{\Xi,\Gamma}(t,s)&B^{\Xi,\Gamma}(t,s)\end{pmatrix},\ \bC(t,s):=\begin{pmatrix}C(t,s)&D(t,s)\\C^{\Xi,\Gamma}(t,s)&D^{\Xi,\Gamma}(t,s)\end{pmatrix},\ (t,s)\in\triangle_2(0,T),
\end{align*}
with
\begin{equation*}
	x^{\Xi,\Gamma}(t):=\Xi(t)x(t)+\int^T_t\Gamma(s,t)x(s)\rd s,\ t\in(t_0,T),
\end{equation*}
and
\begin{equation*}
	f^{\Xi,\Gamma}(t,s):=\Xi(t)f(t,s)+\int^T_t\Gamma(r,t)f(r,s)\rd r,\ (t,s)\in\triangle_2(0,T),
\end{equation*}
for $f=b,\sigma,A,B,C,D$.
\end{lemm}


\begin{proof}
By using the Cauchy--Schwarz inequality, we can easily show that
\begin{equation*}
	\Phi\in L^2_\bF(t_0,T;\bR^{d+\ell}),\ \bA\in L^2(\triangle_2(0,T);\bR^{(d+\ell)\times(d+\ell)})\ \text{and}\ \bC\in\sL^2(\triangle_2(0,T);\bR^{(d+\ell)\times(d+\ell)}).
\end{equation*}

Let $(X^{t_0,x},\Theta^{t_0,x},u^{t_0,x})\in L^2_\bF(t_0,T;\bR^d)\times L^2_{\bF,\mathrm{c}}(\triangle_2(t_0,T);\bR^d)\times\cU(t_0,T)$ be a causal feedback solution to the controlled SVIE \eqref{eq_state} at $(t_0,x)\in\cI$ corresponding to the causal feedback strategy $(\Xi,\Gamma,v)\in\cS(0,T)$. By using the (stochastic) Fubini theorem, we have
\begin{align*}
	u^{t_0,x}(t)&=\Xi(t)X^{t_0,x}(t)+\int^T_t\Gamma(s,t)\Theta^{t_0,x}(s,t)\rd s+v(t)\\
	&=\Xi(t)x(t)+\Xi(t)\int^t_{t_0}\{A(t,s)X^{t_0,x}(s)+B(t,s)u^{t_0,x}(s)+b(t,s)\}\rd s\\
	&\hspace{0.5cm}+\Xi(t)\int^t_{t_0}\{C(t,s)X^{t_0,x}(s)+D(t,s)u^{t_0,x}(s)+\sigma(t,s)\}\rd W(s)\\
	&\hspace{0.5cm}+\int^T_t\Gamma(s,t)x(s)\rd s+\int^T_t\Gamma(s,t)\int^t_{t_0}\{A(s,r)X^{t_0,x}(r)+B(s,r)u^{t_0,x}(r)+b(s,r)\}\rd r\rd s\\
	&\hspace{0.5cm}+\int^T_t\Gamma(s,t)\int^t_{t_0}\{C(s,r)X^{t_0,x}(r)+D(s,r)u^{t_0,x}(r)+\sigma(s,r)\}\rd W(r)\rd s+v(t)\\
	&=x^{\Xi,\Gamma}(t)+\int^t_{t_0}b^{\Xi,\Gamma}(t,s)\rd s+\int^t_{t_0}\sigma^{\Xi,\Gamma}(t,s)\rd W(s)+v(t)\\
	&\hspace{0.5cm}+\int^t_{t_0}(A^{\Xi,\Gamma}(t,s),B^{\Xi,\Gamma}(t,s))\begin{pmatrix}X^{t_0,x}(s)\\u^{t_0,x}(s)\end{pmatrix}\rd s+\int^t_{t_0}(C^{\Xi,\Gamma}(t,s),D^{\Xi,\Gamma}(t,s))\begin{pmatrix}X^{t_0,x}(s)\\u^{t_0,x}(s)\end{pmatrix}\rd W(s)
\end{align*}
for a.e. $t\in(t_0,T)$, a.s. Therefore, the process $\bX^{t_0,x}\in L^2_\bF(t_0,T;\bR^{d+\ell})$ defined by the first equality in \eqref{eq_closed-loop_transform} satisfies the SVIE \eqref{eq_closed-loop_system'}, and the second equality in \eqref{eq_closed-loop_transform} holds.

Conversely, if $\bX^{t_0,x}\in L^2_\bF(t_0,T;\bR^{d+\ell})$ solves the SVIE \eqref{eq_closed-loop_system'}, then the triplet $(X^{t_0,x},\Theta^{t_0,x},u^{t_0,x})\in L^2_\bF(t_0,T;\bR^d)\times L^2_{\bF,\mathrm{c}}(\triangle_2(t_0,T);\bR^d)\times\cU(t_0,T)$ defined by \eqref{eq_closed-loop_transform} is the causal feedback solution of the controlled SVIE \eqref{eq_state}. This completes the proof.
\end{proof}

Now we are ready to prove \cref{theo_SVIE}.


\begin{proof}[Proof of \cref{theo_SVIE}]
By \cref{lemm_SVIE_general}, the SVIE \eqref{eq_closed-loop_system'} admits a unique solution $\bX^{t_0,x}\in L^2_\bF(t_0,T;\bR^{d+\ell})$. By \cref{lemm_closed-loop_transform}, the causal feedback solution $(X^{t_0,x},\Theta^{t_0,x},u^{t_0,x})\in L^2_\bF(t_0,T;\bR^d)\times L^2_{\bF,\mathrm{c}}(\triangle_2(t_0,T);\bR^d)\times\cU(t_0,T)$ to the controlled SVIE \eqref{eq_state} uniquely exists. Furthermore, the following holds:
\begin{equation*}
	\|\bX^{t_0,x}\|_{L^2_\bF(t_0,T)}\leq 2m^2(1+2L)\|\Phi\|_{L^2_\bF(t_0,T)}
\end{equation*}
with $L:=\|\bA\|_{L^2(\triangle_2(0,T))}+\|\bC\|_{\sL^2(\triangle_2(0,T))}$ and $m$ being a number of division intervals $0=U_0<U_1<\cdots<U_m=T$ of $(0,T)$ satisfying
\begin{equation*}
	\Bigl(\int^{U_{i+1}}_{U_i}\!\!\int^t_{U_i}|\bA(t,s)|^2\rd s\!\rd t\Bigr)^{1/2}+\underset{t\in(U_i,U_{i+1})}{\mathrm{ess\,sup}}\Bigl(\int^{U_{i+1}}_t|\bC(s,t)|^2\rd s\Bigr)^{1/2}\leq\frac{1}{2}
\end{equation*}
for $i=0,1,\dots,m-1$. Observe that
\begin{equation*}
	\|\Phi\|_{L^2(t_0,T)}\leq(1+\|\Xi\|_{L^\infty(0,T)}+\|\Gamma\|_{L^2(\triangle_2(0,T))})(\|x\|_{L^2(t_0,T)}+\|b\|_{L^{2,1}_\bF(\triangle_2(t_0,T))}+\|\sigma\|_{L^2_\bF(\triangle_2(t_0,T))})+\|v\|_{L^2_\bF(t_0,T)}.
\end{equation*}
Furthermore, noting the second equality in \eqref{eq_closed-loop_transform} and using Doob's inequality, we have
\begin{align*}
	&\|\Theta^{t_0,x}\|_{L^2_{\bF,\mathrm{c}}(\triangle_2(t_0,T))}=\bE\Bigl[\int^T_{t_0}\sup_{t\in[t_0,s]}|\Theta^{t_0,x}(s,t)|^2\rd s\Bigr]^{1/2}\\
	&\leq\|x\|_{L^2(t_0,T)}+(\|A\|_{L^2(\triangle_2(0,T))}+\|B\|_{L^2(\triangle_2(0,T))})\|\bX^{t_0,x}\|_{L^2_\bF(t_0,T)}+\|b\|_{L^{2,1}_\bF(\triangle_2(t_0,T))}\\
	&\hspace{0.5cm}+2(\|C\|_{\sL^2(\triangle_2(0,T))}+\|D\|_{\sL^2(\triangle_2(0,T))})\|\bX^{t_0,x}\|_{L^2_\bF(t_0,T)}+2\|\sigma\|_{L^2_\bF(\triangle_2(t_0,T))}.
\end{align*}
Therefore, the estimate \eqref{eq_closed-loop_estimate} holds. This completes the proof.
\end{proof}


\begin{rem}\label{rem_universal_constant}
From the above proof, the constant $K>0$ in the estimate \eqref{eq_closed-loop_estimate} can be chosen to be uniform with respect to small scaling of the coefficients. More precisely, for each $(\Xi,\Gamma,v)\in\cS(0,T)$ and $(t_0,x)$, the corresponding causal feedback solution $(X^{t_0,x},\Theta^{t_0,x},u^{t_0,x})$ to the controlled SVIE \eqref{eq_state} with the coefficients $A,B,C,D$ replaced by $\lambda_AA,\lambda_BB,\lambda_CC,\lambda_DD$ for some constants $\lambda_A,\lambda_B,\lambda_C,\lambda_D\in[0,1]$ satisfies the estimate \eqref{eq_closed-loop_estimate} with the same constant $K>0$.
\end{rem}


\subsection{Proof of \cref{theo_Type-II_EBSVIE}: Type-II EBSVIEs}\label{appendix_Type-II_EBSVIE}

First, we prove the well-posedness of a trivial Type-II EBSVIE where $\Xi$ and $\Gamma$ vanish.


\begin{lemm}\label{lemm_Type-II_EBSVIE}
For any $(\psi,\chi)\in\cY:=L^2_\bF(0,T;\bR^d)\times L^{2,1}_\bF(\triangle_2(0,T);\bR^d)$, there exists a unique adapted solution $(\eta,\zeta)\in\cX:=L^2_{\bF,\mathrm{c}}(\triangle_2(0,T);\bR^d)\times L^2_\bF(\triangle_2(0,T);\bR^d)$ to the trivial Type-II EBSVIE
\begin{equation*}
	\begin{dcases}
	\mathrm{d}\eta(t,s)=-\chi(t,s)\rd s+\zeta(t,s)\rd W(s),\ (t,s)\in\triangle_2(0,T),\\
	\eta(t,t)=\psi(t),\ t\in(0,T).
	\end{dcases}
\end{equation*}
Furthermore, the solution map $\cT:(\psi,\chi)\mapsto(\eta,\zeta)$ is a bounded linear operator from $\cY$ to $\cX$.
\end{lemm}


\begin{proof}
We construct the adapted solution. Define $\eta\in L^2_{\bF,\mathrm{c}}(\triangle_2(0,T);\bR^d)$ by
\begin{equation*}
	\eta(t,s):=\bE_s\Bigl[\psi(t)+\int^t_s\chi(t, r)\rd  r\Bigr],\ (t,s)\in\triangle_2(0,T).
\end{equation*}
Also, for a.e.\ $t\in(0,T)$, define $\zeta(t,\cdot)\in L^2_\bF(0,t;\bR^d)$ via the martingale representation theorem:
\begin{equation*}
	\psi(t)+\int^t_0\chi(t, r)\rd  r=\bE\Bigl[\psi(t)+\int^t_0\chi(t, r)\rd  r\Bigr]+\int^t_0\zeta(t,s)\rd W(s).
\end{equation*}
Then we have $\eta(t,t)=\psi(t)$, $t\in(0,T)$, and
\begin{equation*}
	\eta(t,s)=\eta(t,t)+\int^t_s\chi(t, r)\rd  r-\int^t_s\zeta(t, r)\rd W( r),\ (t,s)\in\triangle_2(0,T).
\end{equation*}
Thus, the pair $(\eta,\zeta)\in\cX=L^2_{\bF,\mathrm{c}}(\triangle_2(0,T);\bR^d)\times L^2_\bF(\triangle_2(0,T);\bR^d)$ is an adapted solution to the trivial Type-II EBSVIE. From the construction, we get
\begin{align*}
	\|\zeta\|_{L^2_\bF(\triangle_2(0,T))}&=\bE\Bigl[\int^T_0\!\!\int^t_0|\zeta(t,s)|^2\rd s\!\rd t\Bigr]^{1/2}=\bE\Bigl[\int^T_0\Bigl|\int^t_0\zeta(t,s)\rd W(s)\Bigr|^2\rd t\Bigr]^{1/2}\\
	&\leq\bE\Bigl[\int^T_0\Bigl|\psi(t)+\int^t_0\chi(t,s)\rd s\Bigr|^2\rd t\Bigr]^{1/2}\\
	&\leq\|\psi\|_{L^2_\bF(0,T;\bR^d)}+\|\chi\|_{L^{2,1}_\bF(\triangle_2(0,T))}.
\end{align*}
Also, by Doob's inequality,
\begin{align*}
	\|\eta\|_{L^2_{\bF,\mathrm{c}}(\triangle_2(0,T))}&=\bE\Bigl[\int^T_0\sup_{s\in[0,t]}|\eta(t,s)|^2\rd t\Bigr]^{1/2}=\Bigl(\int^T_0\bE\Bigl[\sup_{s\in[0,t]}|\eta(t,s)|^2\Bigr]\rd t\Bigr)^{1/2}\\
	&=\Bigl(\int^T_0\bE\Bigl[\sup_{s\in[0,t]}\Bigl|\bE_s\Bigl[\psi(t)+\int^t_s\chi(t, r)\rd r\Bigr]\Bigr|^2\Bigr]\rd t\Bigr)^{1/2}\\
	&\leq2\|\psi\|_{L^2_\bF(0,T)}+2\|\chi\|_{L^{2,1}_\bF(\triangle_2(0,T))}.
\end{align*}
The uniqueness of the adapted-solution and the linearity of the solution map $\cT:\cY\to\cX$ are clear. The boundedness of $\cT$ follows from the above estimates. This completes the proof.
\end{proof}


\begin{proof}[Proof of \cref{theo_Type-II_EBSVIE}]
We first prove the duality principle \eqref{eq_duality}. After that, we provide a priori estimate of the adapted solution of Type-II EBSVIE \eqref{eq_Type-II_EBSVIE} by using the duality principle and the estimate \eqref{eq_closed-loop_estimate} of the causal feedback solution to the controlled SVIE. Then, the existence and uniqueness of the adapted solution will follow from \cref{lemm_method_of_continuation}.

\underline{The duality principle \eqref{eq_duality}:} Suppose that $(\eta,\zeta)\in L^2_{\bF,\mathrm{c}}(\triangle_2(0,T);\bR^d)\times L^2_\bF(\triangle_2(0,T);\bR^d)$ is an adapted solution to the Type-II EBSVIE \eqref{eq_Type-II_EBSVIE}. For each $(t_0,x)\in\cI$ and $v\in\cU(0,T)$, denote by $(X,\Theta,u):=(X^{t_0,x},\Theta^{t_0,x},u^{t_0,x})\in L^2_\bF(t_0,T;\bR^d)\times L^2_{\bF,\mathrm{c}}(\triangle_2(t_0,T);\bR^d)\times\cU(t_0,T)$ the causal feedback solution to the controlled SVIE \eqref{eq_state} at $(t_0,x)\in\cI$ corresponding to the causal feedback strategy $(\Xi,\Gamma,v)\in\cS(0,T)$. By the second equality in \eqref{eq_Type-II_EBSVIE}, we have
\begin{equation}\label{eq_psiX}
\begin{split}
	&\bE\Bigl[\int^T_{t_0}\langle\psi(t),X(t)\rangle\rd t\Bigr]\\
	&=\bE\Bigl[\int^T_{t_0}\Bigl\{\langle\eta(t,t),X(t)\rangle-\Bigl\langle\int^T_t(A+B\triangleright\Xi)(s,t)^\top\eta(s,t)\rd s+\int^T_t(C+D\triangleright\Xi)(s,t)^\top\zeta(s,t)\rd s,X(t)\Bigr\rangle\Bigr\}\rd t\Bigr].
\end{split}
\end{equation}
For a.e.\ $t\in(t_0,T)$, noting that $\Theta(t,t)=X(t)$ and applying I\^{o}'s formula to $\langle\eta(t,\cdot),\Theta(t,\cdot)\rangle$ on $(t_0,t)$, we have
\begin{align*}
	&\bE[\langle\eta(t,t),X(t)\rangle]=\bE[\langle\eta(t,t),\Theta(t,t)\rangle]\\
	&=\langle\bE[\eta(t,t_0)],x(t)\rangle\\
	&\hspace{0.5cm}+\bE\Bigl[\int^t_{t_0}\Bigl\{-\Bigl\langle\chi(t,s)+\Gamma(t,s)^\top\int^T_sB(r,s)^\top\eta(r,s)\rd r+\Gamma(t,s)^\top\int^T_sD(r,s)^\top\zeta(r,s)\rd r,\Theta(t,s)\Bigr\rangle\\
	&\hspace{3cm}+\langle\eta(t,s),A(t,s)X(s)+B(t,s)u(s)+b(t,s)\rangle\\
	&\hspace{3cm}+\langle\zeta(t,s),C(t,s)X(s)+D(t,s)u(s)+\sigma(t,s)\rangle\Bigr\}\rd s\Bigr]\\
	&=\langle\bE[\eta(t,t_0)],x(t)\rangle+\bE\Bigl[\int^t_{t_0}\{-\langle\chi(t,s),\Theta(t,s)\rangle+\langle\eta(t,s),b(t,s)\rangle+\langle\zeta(t,s),\sigma(t,s)\rangle\}\rd s\Bigr]\\
	&\hspace{0.5cm}+\bE\Bigl[\int^t_{t_0}\langle (A+B\triangleright\Xi)(t,s)^\top\eta(t,s)+(C+D\triangleright\Xi)(t,s)^\top\zeta(t,s),X(s)\rangle\rd s\Bigr]\\
	&\hspace{0.5cm}+\bE\Bigl[\int^t_{t_0}\Bigl\{\langle B(t,s)^\top\eta(t,s)+D(t,s)^\top\zeta(t,s),u(s)-\Xi(s)X(s)\rangle\\
	&\hspace{3cm}-\Bigl\langle\int^T_s\{B(r,s)^\top\eta(r,s)+D(r,s)^\top\zeta(r,s)\}\rd r,\Gamma(t,s)\Theta(t,s)\Bigr\rangle\Bigr\}\rd s\Bigr].
\end{align*}
Integrating both sides with respect to $t\in(t_0,T)$ and applying Fubini's theorem, we get
\begin{equation}\label{eq_etaX}
\begin{split}
	&\bE\Bigl[\int^T_{t_0}\langle\eta(t,t),X(t)\rangle\rd t\Bigr]\\
	&=\int^T_{t_0}\langle\bE[\eta(t,t_0)],x(t)\rangle\rd t\\
	&\hspace{0.5cm}+\bE\Bigl[\int^T_{t_0}\Bigl\{-\int^T_t\langle\chi(s,t),\Theta(s,t)\rangle\rd s+\int^T_t\langle\eta(s,t),b(s,t)\rangle\rd s+\int^T_t\langle\zeta(s,t),\sigma(s,t)\rangle\rd s\Bigr\}\rd t\Bigr]\\
	&\hspace{0.5cm}+\bE\Bigl[\int^T_{t_0}\Bigl\langle\int^T_t(A+B\triangleright\Xi)(s,t)^\top\eta(s,t)\rd s+\int^T_t(C+D\triangleright\Xi)(s,t)^\top\zeta(s,t)\rd s,X(t)\Bigr\rangle\rd t\Bigr]\\
	&\hspace{0.5cm}+\bE\Bigl[\int^T_{t_0}\Bigl\langle\int^T_t\{B(s,t)^\top\eta(s,t)+D(s,t)^\top\zeta(s,t)\}\rd s,u(t)-\Xi(t)X(t)-\int^T_t\Gamma(r,t)\Theta(r,t)\rd r\Bigr\rangle\rd t\Bigr]
\end{split}
\end{equation}
By \eqref{eq_psiX}, \eqref{eq_etaX} and the relation $u(t)=\Xi(t)X(t)+\int^T_t\Gamma(r,t)\Theta(r,t)\rd r+v(t)$, we obtain the duality principle \eqref{eq_duality}.

\underline{Existence and uniqueness of the adapted solution:} We define a map $F$ from $\cX:=L^2_{\bF,\mathrm{c}}(\triangle_2(0,T);\bR^d)\times L^2_\bF(\triangle_2(0,T);\bR^d)$ to $\cY:=L^2_\bF(0,T;\bR^d)\times L^{2,1}_\bF(\triangle_2(0,T);\bR^d)$ by
\begin{align*}
	&F(\eta,\zeta):=(F_1(\eta,\zeta),F_2(\eta,\zeta)),\\
	&F_1(\eta,\zeta)(t):=\int^T_t(A+B\triangleright\Xi)(r,t)^\top\eta(r,t)\rd r+\int^T_t(C+D\triangleright\Xi)(r,t)^\top\zeta(r,t)\rd r,\ t\in(0,T),\\
	&F_2(\eta,\zeta)(t,s):=\Gamma(t,s)^\top\int^T_sB(r,s)^\top\eta(r,s)\rd r+\Gamma(t,s)^\top\int^T_sD(r,s)^\top\zeta(r,s)\rd r,\ (t,s)\in\triangle_2(0,T),
\end{align*}
for $(\eta,\zeta)\in\cX$. It is easy to see that $F:\cX\to\cY$ is a bounded linear operator. Also, consider the bounded linear operator $\cT:\cY\to\cX$ defined in \cref{lemm_Type-II_EBSVIE}. Observe that, for each $(\psi,\chi)\in\cY$, $(\eta,\zeta)\in\cX$ is an adapted solution to the Type-II EBSVIE \eqref{eq_Type-II_EBSVIE} if and only if the following holds:
\begin{equation*}
	(\eta,\zeta)=\cT(F(\eta,\zeta)+(\psi,\chi))
\end{equation*}
In order to apply \cref{lemm_method_of_continuation}, we have to show the a priori estimate.

\underline{A priori estimate of the adapted solution:} Let $\lambda\in[0,1]$ be fixed. For each $(\psi_i,\chi_i)\in\cY$ with $i=1,2$, assume that $(\eta_i,\zeta_i)\in\cX$ satisfy the equations
\begin{equation*}
	(\eta_i,\zeta_i)=\cT(\lambda F(\eta_i,\zeta_i)+(\psi_i,\chi_i)),\ i=1,2.
\end{equation*}
Define $(\bar{\eta},\bar{\zeta}):=(\eta_1-\eta_2,\zeta_1-\zeta_2)\in\cX$ and $(\bar{\psi},\bar{\chi}):=(\psi_1-\psi_2,\chi_1-\chi_2)\in\cY$. We will show the estimate
\begin{equation*}
	\|(\bar{\eta},\bar{\zeta})\|_\cX\leq K\|(\bar{\psi},\bar{\chi})\|_\cY.
\end{equation*}
Here and the rest of this proof, $K>0$ denotes a universal constant which does not depend on $\lambda$ or $(\psi_i,\chi_i)$. Observe that $(\bar{\eta},\bar{\zeta})$ is an adapted solution to the Type-II EBSVIE \eqref{eq_Type-II_EBSVIE} with the coefficients $A,B,C,D$ replaced by $\lambda A,\lambda B,\lambda C,\lambda D$ and $(\psi,\chi)$ replaced by $(\bar{\psi},\bar{\chi})$. Let $(X_1,\Theta_1)\in L^2_\bF(0,T;\bR^d)\times L^2_{\bF,\mathrm{c}}(\triangle_2(0,T);\bR^d)$ be the causal feedback solution to the homogeneous controlled SVIE \eqref{eq_state_0} with the coefficients $(\lambda A,\lambda B,\lambda C,\lambda D)$, the zero input condition $(t_0,x)=(0,0)\in\cI$, and a causal feedback strategy $(\Xi,\Gamma,v)\in\cS(0,T)$ with $v\in\cU(0,T)$ being arbitrary. By the duality principle \eqref{eq_duality}, we have
\begin{align*}
	&\bE\Bigl[\int^T_0\Bigl\langle\int^T_t\{\lambda B(s,t)^\top\bar{\eta}(s,t)+\lambda D(s,t)^\top\bar{\zeta}(s,t)\}\rd s,v(t)\Bigr\rangle\rd t\Bigr]\\
	&=\bE\Bigl[\int^T_0\Bigl\{\langle\bar{\psi}(t),X_1(t)\rangle+\int^T_t\langle\bar{\chi}(s,t),\Theta_1(s,t)\rangle\rd s\Bigr\}\rd t\Bigr].
\end{align*}
By \cref{theo_SVIE}, noting \cref{rem_universal_constant}, the right hand side of the above equality is estimated as
\begin{align*}
	&\Bigl|\bE\Bigl[\int^T_0\Bigl\{\langle\bar{\psi}(t),X_1(t)\rangle+\int^T_t\langle\bar{\chi}(s,t),\Theta_1(s,t)\rangle\rd s\Bigr\}\rd t\Bigr]\Bigr|\\
	&\leq\|\bar{\psi}\|_{L^2_\bF(0,T)}\|X_1\|_{L^2_\bF(0,T)}+\|\bar{\chi}\|_{L^{2,1}_\bF(\triangle_2(0,T))}\|\Theta_1\|_{L^2_{\bF,\mathrm{c}}(\triangle_2(0,T))}\\
	&\leq K\|(\bar{\psi},\bar{\chi})\|_\cY\|v\|_{L^2_\bF(0,T)}.
\end{align*}
Therefore, we have the following estimate:
\begin{align*}
	&\bE\Bigl[\int^T_0\Bigl|\int^T_t\{\lambda B(s,t)^\top\bar{\eta}(s,t)+\lambda D(s,t)^\top\bar{\zeta}(s,t)\}\rd s\Bigr|^2\rd t\Bigr]^{1/2}\\
	&=\sup_{\substack{v\in L^2_\bF(0,T;\bR^\ell)\\\|v\|_{L^2_\bF(0,T)}\leq1}}\Bigl|\bE\Bigl[\int^T_0\Bigl\langle\int^T_t\{\lambda B(s,t)^\top\bar{\eta}(s,t)+\lambda D(s,t)^\top\bar{\zeta}(s,t)\}\rd s,v(t)\Bigr\rangle\rd t\Bigr]\Bigr|\\
	&\leq K\|(\bar{\psi},\bar{\chi})\|_\cY.
\end{align*}
By using the Cauchy–Schwarz inequality, we get
\begin{align*}
	\|\lambda F_2(\bar{\eta},\bar{\zeta})\|_{L^{2,1}_\bF(\triangle_2(0,T))}&\leq\|\Gamma\|_{L^2(\triangle_2(0,T))}\bE\Bigl[\int^T_0\Bigl|\int^T_t\{\lambda B(s,t)^\top\bar{\eta}(s,t)+\lambda D(s,t)^\top\bar{\zeta}(s,t)\}\rd s\Bigr|^2\rd t\Bigr]^{1/2}\\
	&\leq K\|(\bar{\psi},\bar{\chi})\|_\cY.
\end{align*}
Next, for a given $\varphi\in L^2_\bF(0,T;\bR^d)$, let $(X_2,\Theta_2)\in L^2_\bF(0,T;\bR^d)\times L^2_{\bF,\mathrm{c}}(\triangle_2(0,T);\bR^d)$ be the causal feedback solution to the controlled SVIE \eqref{eq_state} with the coefficients $(\lambda A,\lambda B,\lambda C,\lambda D)$, the inhomogeneous terms
\begin{equation*}
	b(t,s)=\lambda(A+B\triangleright\Xi)(t,s)\varphi(s),\ \sigma(t,s)=\lambda(C+D\triangleright\Xi)(t,s)\varphi(s),\ (t,s)\in\triangle_2(0,T),
\end{equation*}
the input condition $(t_0,x)=(0,0)$, and the causal feedback strategy $(\Xi,\Gamma,0)\in\cS(0,T)$. Then the duality principle \eqref{eq_duality} yields that
\begin{align*}
	&\bE\Bigl[\int^T_0\Bigl\langle\int^T_t\{\lambda(A+B\triangleright\Xi)(s,t)^\top\bar{\eta}(s,t)+\lambda(C+D\triangleright\Xi)(s,t)^\top\bar{\zeta}(s,t)\}\rd s,\varphi(t)\Bigr\rangle\rd t\Bigr]\\
	&=\bE\Bigl[\int^T_{t_0}\Bigl\{\int^T_t\langle\bar{\eta}(s,t),b(s,t)\rangle\rd s+\int^T_t\langle\bar{\zeta}(s,t),\sigma(s,t)\rangle\rd s\Bigr\}\rd t\Bigr]\\
	&=\bE\Bigl[\int^T_0\Bigl\{\langle\bar{\psi}(t),X_2(t)\rangle+\int^T_t\langle\bar{\chi}(s,t),\Theta_2(s,t)\rangle\rd s\Bigr\}\rd t\Bigr].
\end{align*}
Noting that $\|b\|_{L^{2,1}_\bF(\triangle_2(0,T))}+\|\sigma\|_{L^2_\bF(\triangle_2(0,T))}\leq K\|\varphi\|_{L^2_\bF(0,T)}$, by the same arguments as above, we obtain
\begin{align*}
	\|\lambda F_1(\bar{\eta},\bar{\zeta})\|_{L^2_\bF(0,T)}&=\bE\Bigl[\int^T_0\Bigl|\int^T_t\{\lambda(A+B\triangleright\Xi)(s,t)^\top\bar{\eta}(s,t)+\lambda(C+D\triangleright\Xi)(s,t)^\top\bar{\zeta}(s,t)\}\rd s\Bigr|^2\rd t\Bigr]^{1/2}\\
	&\leq K\|(\bar{\psi},\bar{\chi})\|_\cY.
\end{align*}
From the above estimates, we have
\begin{equation*}
	\|\lambda F(\bar{\eta},\bar{\zeta})\|_\cY\leq K\|(\bar{\psi},\bar{\chi})\|_\cY,
\end{equation*}
and hence the desired a priori estimate holds:
\begin{equation*}
	\|(\bar{\eta},\bar{\zeta})\|_\cX=\|\cT(\lambda F(\bar{\eta},\bar{\zeta})+(\bar{\psi},\bar{\chi}))\|_\cX\leq K\|(\bar{\psi},\bar{\chi})\|_\cY.
\end{equation*}

By \cref{lemm_method_of_continuation}, for any $(\psi,\chi)\in\cY=L^2_\bF(0,T;\bR^d)\times L^{2,1}_\bF(\triangle_2(0,T);\bR^d)$, there exists a unique adapted solution $(\eta,\zeta)\in\cX=L^2_{\bF,\mathrm{c}}(\triangle_2(0,T);\bR^d)\times L^2_\bF(\triangle_2(0,T);\bR^d)$ to the Type-II EBSVIE \eqref{eq_Type-II_EBSVIE}. This completes the proof.
\end{proof}


\subsection{Proof of \cref{theo_Lyapunov--Volterra}: Lyapunov--Volterra equations}\label{appendix_Lyapunov--Volterra}

We prove \cref{theo_Lyapunov--Volterra}. The representation formula \eqref{eq_quadratic_representation} follows from \cref{cor_Ito_feedback}. The idea of the proof of the well-posedness of the Lyapunov--Volterra equation \eqref{eq_Lyapunov--Volterra} is similar to the one of \cref{theo_Type-II_EBSVIE}.


\begin{proof}[Proof of \cref{theo_Lyapunov--Volterra}]
\underline{The representation formula \eqref{eq_quadratic_representation}:} By \cref{cor_Ito_feedback}, it is trivial that if $P=(P^{(1)},P^{(2)})\in\Pi(0,T)$ satisfies the Lyapunov--Volterra equation \eqref{eq_Lyapunov--Volterra}, then the representation formula \eqref{eq_quadratic_representation} holds.

\underline{Existence and uniqueness of the solution:} Let $\cX:=\Pi(0,T)$ and
\begin{equation*}
	\cY:=L^\infty(0,T;\bS^d)\times L^2(\triangle_2(0,T);\bR^{d\times d})\times L^{2,2,1}_\sym(\square_3(0,T);\bR^{d\times d}).
\end{equation*}
Define $\cT:\cY\to\cX$ by $\cT(Q):=(\cT_1(Q),\cT_2(Q))$ with
\begin{align*}
	&\cT_1(Q)(t):=Q^{(1)}(t),\ t\in(0,T),\\
	&\cT_2(Q)(s_1,s_2,t):=Q^{(2)}(s_1,s_2)\1_{\triangle_2(0,T)}(s_1,s_2)+Q^{(2)}(s_2,s_1)^\top\1_{\triangle_2(0,T)}(s_2,s_1)-\int^{s_1\wedge s_2}_tQ^{(3)}(s_1,s_2,\tau)\rd\tau,\\
	&\hspace{6cm}(s_1,s_2,t)\in\square_3(0,T),
\end{align*}
for $Q=(Q^{(1)},Q^{(2)},Q^{(3)})\in\cY$. Clearly, $\cT:\cY\to\cX$ is a bounded linear operator. Also, define $F:\cX\to\cY$ by
\begin{align*}
	F(P):=(F^{(1)}[\Xi;P],F^{(2)}[\Xi,\Gamma;P],F^{(3)}[\Gamma;P])
\end{align*}
for $P=(P^{(1)},P^{(2)})\in\cX=\Pi(0,T)$. By \cref{lemm_lint-rint_integrability}, $F:\cX\to\cY$ is a bounded linear operator. For each $Q=(Q^{(1)},Q^{(2)},Q^{(3)})\in\cY$, $P=(P^{(1)},P^{(2)})\in\cX$ is a solution to the Lyapunov--Volterra equation \eqref{eq_Lyapunov--Volterra} if and only if the following equation holds:
\begin{equation*}
	P=\cT(F(P)+Q).
\end{equation*}
In order to apply \cref{lemm_method_of_continuation}, we have to show the a priori estimate.

\underline{A priori estimate of the solution:} Let $\lambda\in[0,1]$ be fixed. For each $Q_i=(Q^{(1)}_i,Q^{(2)}_i,Q^{(3)}_i)\in\cY$ with $i=1,2$, assume that $P_i=(P^{(1)}_i,P^{(2)}_i)\in\cX$ satisfy the equations
\begin{equation*}
	P_i=\cT(\lambda F(P_i)+Q_i),\ i=1,2.
\end{equation*}
Define $\bar{P}:=(\bar{P}^{(1)},\bar{P}^{(2)}):=(P^{(1)}_1-P^{(1)}_2,P^{(2)}_1-P^{(2)}_2)\in\cX$ and $\bar{Q}:=(\bar{Q}^{(1)},\bar{Q}^{(2)},\bar{Q}^{(3)}):=(Q^{(1)}_1-Q^{(1)}_2,Q^{(2)}_1-Q^{(2)}_2,Q^{(3)}_1-Q^{(3)}_2)\in\cY$. We will show the estimate
\begin{equation*}
	\|\bar{P}\|_\cX\leq K\|\bar{Q}\|_\cY.
\end{equation*}
Here and the rest of this proof, $K>0$ denotes a universal constant which does not depend on $\lambda$ or $Q_i$. Observe that $\bar{P}=(\bar{P}^{(1)},\bar{P}^{(2)})$ is a solution to the Lyapunov--Volterra equation \eqref{eq_Lyapunov--Volterra} with the coefficients $A,B,C,D$ replaced by $\lambda A,\lambda B,\sqrt{\lambda}C,\sqrt{\lambda}D$ and $Q=(Q^{(1)},Q^{(2)},Q^{(3)})$ replaced by $\bar{Q}=(\bar{Q}^{(1)},\bar{Q}^{(2)},\bar{Q}^{(3)})$. For each $t_0\in[0,T)$, define $\bar{\cP}^{t_0}:L^2(t_0,T;\bR^d)\to L^2(t_0,T;\bR^d)$ by
\begin{equation*}
	\bar{\cP}^{t_0}:=\bar{P}^{(1)}(t)x(t)+\int^T_{t_0}\bar{P}^{(2)}(t, r,t_0)x( r)\rd r,\ t\in(t_0,T),
\end{equation*}
for $x\in L^2(t_0,T;\bR^d)$. By \cref{lemm_self-adjoint}, $\bar{\cP}^{t_0}$ is a self-adjoint bounded linear operator on the Hilbert space $L^2(t_0,T;\bR^d)$. Now let $(X,\Theta)\in L^2_\bF(t_0,T;\bR^d)\times L^2_{\bF,\mathrm{c}}(\triangle_2(t_0,T);\bR^d)$ be the causal feedback solution to the homogeneous controlled SVIE \eqref{eq_state_0} with the coefficients $(\lambda A,\lambda B,\sqrt{\lambda}C,\sqrt{\lambda}D)$, a given input condition $(t_0,x)\in\cI$, and the causal feedback strategy $(\Xi,\Gamma,0)\in\cS(0,T)$. By the representation formula \eqref{eq_quadratic_representation}, we have
\begin{align*}
	\langle\bar{\cP}^{t_0}x,x\rangle_{L^2(t_0,T)}&=\int^T_{t_0}\langle \bar{P}^{(1)}(t)x(t),x(t)\rangle\rd t+\int^T_{t_0}\!\!\int^T_{t_0}\langle \bar{P}^{(2)}(s_1,s_2,t_0)x(s_2),x(s_1)\rangle\rd s_1\!\rd s_2\\
	&=\bE\Bigl[\int^T_{t_0}\Bigl\{\langle \bar{Q}^{(1)}(t)X(t),X(t)\rangle+2\int^T_t\langle \bar{Q}^{(2)}(s,t)X(t),\Theta(s,t)\rangle\rd s\\
	&\hspace{4cm}+\int^T_t\!\!\int^T_t\langle \bar{Q}^{(3)}(s_1,s_2,t)\Theta(s_2,t),\Theta(s_1,t)\rangle\rd s_1\!\rd s_2\Bigr\}\rd t\Bigr].
\end{align*}
By \cref{theo_SVIE}, noting \cref{rem_universal_constant}, the right hand side of the above equality is estimated as
\begin{align*}
	&\Bigl|\bE\Bigl[\int^T_{t_0}\Bigl\{\langle \bar{Q}^{(1)}(t)X(t),X(t)\rangle+2\int^T_t\langle \bar{Q}^{(2)}(s,t)X(t),\Theta(s,t)\rangle\rd s\\
	&\hspace{4cm}+\int^T_t\!\!\int^T_t\langle \bar{Q}^{(3)}(s_1,s_2,t)\Theta(s_2,t),\Theta(s_1,t)\rangle\rd s_1\!\rd s_2\Bigr\}\rd t\Bigr]\Bigr|\\
	&\leq\|\bar{Q}^{(1)}\|_{L^\infty(0,T)}\|X\|^2_{L^2_\bF(t_0,T)}+2\|\bar{Q}^{(2)}\|_{L^2(\triangle_2(0,T))}\|X\|_{L^2_\bF(t_0,T)}\|\Theta\|_{L^2_{\bF,\mathrm{c}}(\triangle_2(t_0,T))}\\
	&\hspace{4cm}+\|\bar{Q}^{(3)}\|_{L^{2,2,1}_\sym(\square_3(0,T))}\|\Theta\|^2_{L^2_{\bF,\mathrm{c}}(\triangle_2(t_0,T))}\\
	&\leq K\|\bar{Q}\|_\cY\|x\|^2_{L^2(t_0,T)}.
\end{align*}
Therefore, the operator norm $\|\bar{\cP}^{t_0}\|_\op$ of $\bar{\cP}^{t_0}$ is estimated as
\begin{equation*}
	\|\bar{\cP}^{t_0}\|_\op=\sup_{\substack{x\in L^2(t_0,T;\bR^d)\\\|x\|_{L^2(t_0,T)}\leq 1}}|\langle\bar{\cP}^{t_0}x,x\rangle_{L^2(t_0,T)}|\leq K\|\bar{Q}\|_\cY.
\end{equation*}
We note that, for any $M=(M_1,\dots,M_{d_1})\in L^2(\triangle_2(0,T);\bR^{d\times d})$, $\bar{P}\rint M\in L^2(\triangle_2(0,T);\bR^{d\times d})$ is written by
\begin{equation*}
	(\bar{P}\rint M)(s,t)=((\bar{\cP}^tM_1(\cdot,t))(s),\dots,(\bar{\cP}^tM_d(\cdot,t))(s)),\ (s,t)\in\triangle_2(0,T),
\end{equation*}
and thus
\begin{equation*}
	\|\bar{P}\rint M\|_{L^2(\triangle_2(0,T))}\leq K\|\bar{Q}\|_\cY\|M\|_{L^2(\triangle_2(0,T))}.
\end{equation*}
Similarly, for any $M=(M_1,\dots,M_d)\in \sL^2(\triangle_2(0,T);\bR^{d\times d})$ and $N=(N_1,\dots,N_d)\in \sL^2(\triangle_2(0,T);\bR^{d\times d})$, $M^\top\lint\bar{P}\rint N\in L^\infty(0,T;\bR^{d\times d})$ is written by
\begin{equation*}
	(M^\top\lint \bar{P}\rint N)(t)=(\langle M_k(\cdot,t),\bar{\cP}^t N_\ell(\cdot,t)\rangle_{L^2(t,T)})_{k,\ell},\ t\in(0,T),
\end{equation*}
and thus
\begin{equation*}
	\|M^\top\lint \bar{P}\rint N\|_{L^\infty(0,T)}\leq K\|\bar{Q}\|_\cY\|M\|_{\sL^2(\triangle_2(0,T))}\|N\|_{\sL^2(\triangle_2(0,T))}.
\end{equation*}
Noting the definition \eqref{eq_Lyapunov--Volterra_coefficients} of the coefficients $(F^{(1)},F^{(2)},F^{(3)})$, we see that
\begin{equation*}
	\|F(\bar{P})\|_\cY=\|(F^{(1)}[\Xi;\bar{P}],F^{(2)}[\Xi,\Gamma;\bar{P}],F^{(3)}[\Gamma;\bar{P}])\|_\cY\leq K\|\bar{Q}\|_\cY,
\end{equation*}
and thus the desired a priori estimate holds:
\begin{equation*}
	\|\bar{P}\|_\cX=\|\cT(\lambda F(\bar{P})+\bar{Q})\|_\cX\leq K\|\bar{Q}\|_\cY.
\end{equation*}

By \cref{lemm_method_of_continuation}, for any $Q=(Q^{(1)},Q^{(2)},Q^{(3)})\in\cY=L^\infty(0,T;\bS^d)\times L^2(\triangle_2(0,T);\bR^{d\times d})\times L^{2,2,1}_\sym(\square_3(0,T);\bR^{d\times d})$, there exists a unique solution $P=(P^{(1)},P^{(2)})\in\cX=\Pi(0,T)$ to the Lyapunov--Volterra equation \eqref{eq_Lyapunov--Volterra}. This completes the proof.
\end{proof}


\end{document}